\documentclass[11pt, reqno, a4paper]{amsart}

\usepackage{amssymb, amsmath, amscd}
\usepackage{amsthm}

\usepackage{verbatim}

\usepackage{color}

\usepackage{etex}
\usepackage{dsfont}
\usepackage[matrix, arrow, curve, frame, arc]{xy}
\usepackage{pgf}
\usepackage{tikz-cd}
\usetikzlibrary{arrows,shapes,automata,backgrounds,petri,calc}
\usepackage[square,sort,comma,numbers]{natbib}
 \usepackage{url}
 \usepackage{hyperref}
 \usepackage[shortlabels]{enumitem} 

\newcommand{\tikzAngleOfLine}{\tikz@AngleOfLine}
\def\tikz@AngleOfLine(#1)(#2)#3{%
\pgfmathanglebetweenpoints{%
\pgfpointanchor{#1}{center}}{%
\pgfpointanchor{#2}{center}}
\pgfmathsetmacro{#3}{\pgfmathresult}%
}

\newcommand{\bN}{\mathbb{N}}

\newcommand{\bZ}{\mathbb{Z}}

\newcommand{\bP}{\mathbb{P}}

\newcommand{\wt}{\widetilde}
\newcommand{\ba}{\bar{\alpha}}
\newcommand{\La}{\Lambda}

\newcommand{\ve}{\varepsilon}

\newcommand{\cM}{\mathcal{M}}
\newcommand{\cK}{\mathcal{K}}

\newcommand{\cD}{\mathcal{D}}

\newcommand{\rad}{\operatorname{rad}}
\newcommand{\soc}{\operatorname{soc}}
\newcommand{\add}{\operatorname{add}}
\renewcommand{\ker}{\operatorname{Ker}}
\newcommand{\ima}{\operatorname{Im}} 
\newcommand{\cok}{\operatorname{Coker}} 
\newcommand{\silt}{\operatorname{silt}} 

\renewcommand{\mod}{\operatorname{mod}}
\newcommand{\op}{\operatorname{op}}
\newcommand{\Hom}{\operatorname{Hom}} 
\newcommand{\proj}{\operatorname{proj}}
\newcommand{\inj}{\operatorname{inj}}
\newcommand{\End}{\operatorname{End}}

\def\vec#1{\left[\begin{smallmatrix}#1\end{smallmatrix}\right]}

\begin{document}

\newtheorem{defi}{Definition}[section]
\newtheorem{rem}[defi]{Remark}
\newtheorem{prop}[defi]{Proposition}
\newtheorem{ques}[defi]{Question}
\newtheorem{lemma}[defi]{Lemma}
\newtheorem{cor}[defi]{Corollary}
\newtheorem{thm}[defi]{Theorem}
\newtheorem{expl}[defi]{Example} 
\newtheorem*{mthm}{Main Theorem}

\parindent0pt

\title[Iterated mutations]{Iterated mutations of symmetric periodic algebras$^*$ \footnote{$^*$\tiny This research has been supported from the research grant no. 2023/51/D/ST1/01214 of the Polish National Science Center.} }

\author[ ]{Adam Skowyrski}
\address[Adam Skowyrski]{Faculty of Mathematics and Computer Science, 
Nicolaus Copernicus University, Chopina 12/18, 87-100 Torun, Poland}
\email{skowyr@mat.umk.pl}

\subjclass[2020]{Primary: 16D50, 16E05, 18G80, 16E35}
\keywords{Symmetric algebra, Periodic algebra, Derived equivalence, Mutation} 

\begin{abstract} Following methods used by A. Dugas \cite{Du} for investigating derived equivalent pairs of 
(weakly) symmetric algebras, we apply them in a specific situation, obtaining new deep results concerning 
iterated mutations of symmetric periodic algebras. More specifically, for any symmetric algebra $\Lambda$, 
and an arbitrary vertex $i$ of its Gabriel quiver, one can define mutation $\mu_i(\Lambda)$ of $\Lambda$ at 
vertex $i$ via silting mutation of the stalk complex $\La$. Then $\mu_i(\Lambda)$ is again symmetric, and we 
can iterate this process. We want to understand the order of $\mu_i$, in case the vertex $i$ is $d$-periodic, i.e. 
the simple module $S_i$ associated to $i$ is periodic of period $d$ (with respect to the syzygy). The main result 
of this paper shows that then $\mu_i$ has order $d-2$, that is $\mu_i^{d-2}(\Lambda)\cong\Lambda$ (modulo socle), 
under some additional assumption on the (periodic) projective-injective resolution of $S_i$. Besides, we present briefly 
some consequences concerning arbitrary periodic vertex and give few sugestive examples showing that this property 
should hold in general, i.e. without restrictions on the periodic projective resolution. \end{abstract}

\maketitle

\section{Introduction and the main result}\label{sec:1} 

Throughout the paper, by an algebra we mean a basic finite-dimensional algebra over an algebraically 
closed field $K$. For simplicity, we assume that all algebras are connected. We denote by $\mod\La$ 
the category of finitely generated right $\La$-modules, and by $\proj\La$ its full subcategory formed by 
projective modules. \smallskip  

Every algebra admits a presentation, 
that is, we have an isomorphism $\La\cong KQ/I$, where $KQ$ is the path algebra of a quiver $Q=(Q_0,Q_1,s,t)$, and 
$I$ is an ideal generated by a finite number of so called relations, i.e. elements of $KQ$ being combinations of paths 
of length $\geqslant 2$ with common source and target (see Section \ref{sec:2} for more details). The quiver 
$Q$ is determined uniquely, up to permutation of vertices, and it is called the Gabriel quiver $Q_\La$ of 
an algebra $\La$. \smallskip 

For an algebra $\La=KQ/I$, we have a complete set $e_i$, $i\in Q_0=\{1,\dots,n\}$, of primitive idempotents, 
which induce the decomposition of $\La=e_1\La\oplus\dots e_n\La$, into a direct sum of all indecomposable projective 
modules $P_i=e_i\La$ in $\mod\La$. Dually, we have an associated decomposition of $D(\La)=D(\La e_1)\oplus \dots 
\oplus D(\La e_n)$ into a direct sum of all indecomposable injective modules $I_i=D(\La e_i)$ in $\mod\La$, where 
$$D=\Hom_K(-,K):\mod A^{\op}\to \mod \La$$ 
is the standard duality on $\mod\La$. By $S_i$ we denote the simple module $S_i=P_i/\rad P_i\simeq\soc(I_i)$ in $\mod\La$ 
associated to a vertex $i\in Q_0$. \smallskip 

A prominent class of algebras is formed by the self-injective algebras, for which $\La$ is injective as a $\La$-module, 
or equivalently, all projective and injective modules in $\mod\La$ coincide. In particular, then $I_i\simeq P_{\nu(i)}$, 
for the Nakayama permutation $\nu$, and hence, the top $S_i$ of $P_i$ is also the socle of $P_{\nu(i)}$. In case 
$\La$ is a symmetric algebra, that is, there exists an associative non-degenerate symmetric $K$-bilinear form 
$(-,-):\La\times\La\to K$, the permutation $\nu$ is identity, so we have isomorphisms $I_i\simeq P_i$ and 
$\soc(P_i)=\soc(I_i)=S_i$. In this article, we will focus on the symmetric algebras, although all the results 
work in more general setup, that is, for weakly symmetric algebras (i.e. satisfying $P_i\simeq I_i$, for all $i\in Q_0$). \smallskip 

Two self-injective algebras $\La,\La'$ are said to be socle equivalent if and only if the quotient algebras 
$\La/\soc(\La)$ and $\La'/\soc(\La')$ are isomorphic. In this case, we write $\La\sim\La'$. For $\La,\La'$ being 
symmetric, every $\soc(P_i)$ is generated by an element $\omega\in e_i\La e_i$, and $\La\sim\La'$ if and only if 
$\La$ and $\La'$ have the same presentations, except relations in $e_i\La e_i$ and $e_i\La' e_i$ may differ by a 
socle summand $\lambda\omega$, $\lambda\in K$. \medskip 

Another important algebras for this paper are the periodic algebras. Recall that, for a module $X$ in $\mod\La$, 
its syzygy $\Omega(X)$ is the kernel of an arbitrary projective cover of $X$ in $\mod\La$. A 
module $X$ in $\mod\La$ is called periodic, provided that $\Omega^d(X)\simeq X$, for some $d\geqslant 1$, and the 
smallest such $d$ is called the period of $X$. We will sometimes say that then $X$ is $d$-periodic. \smallskip 

By a periodic algebra we mean an algebra $\La$, which is periodic as a $\La$-bimodule, or equivalently, as a 
module over its enveloping algebra $\La^e=\La^{\op}\otimes_K\La$. Note that every periodic algebra $\La$ of 
period $d$ has periodic module category, that is, all non-projective modules in $\mod\La$ are periodic (with 
period dividing the period of $\La$ \cite[see Theorem IV.11.19]{SkY}). According to the results of \cite{GSS}, 
we know that for a periodic algebra $\La$, all simple modules in $\mod\La$ are periodic, so in particular, 
then $\La$ must be self-injective. \bigskip 

Given an algebra, we denote by $\cK^b_\La$ the homotopy category of bonded complexes of modules in $\proj\La$, 
which is a triangulated category with the suspension functor given by left shift $(-)[]$ (see \cite{Hap}).  
Recall \cite{Ai} that a complex $T$ in $\cK^b_\La$ is called silting, if it generates $\cK^b_\La$, and there are 
no non-zero morphisms from $T$ to its positive shifts $T[i]$, $i>0$. A complex $T$ is said to be basic if 
it is a direct sum of mutually non-isomorphic indecomposable objects. \smallskip 

Following \cite{Ai}, for any basic silting complex $T=X\oplus Y$ in $\cK^b_\La$, one can define its mutation $\mu_X(T)$, which 
is another silting complex called the mutation of $T$ with respect to $X$. Namely, $\mu_X(T)=X'\oplus Y\in\cK^b_\La$, where 
$X'$ is determined by a left $\add Y$-approximation $f:X\to Y'$ of $X$ and a triangle in $\cK^b_\La$ of the form 
$$\xymatrix{X\ar[r]^{f} & Y' \ar[r] & X' \ar[r] & X[1]}$$ 
(more details in Section \ref{sec:mut}). If $X$ is indecomposable, we say it is an irreducible mutation. Silting 
theory in general, deals with subcategories of triangulated categories, but we omit this context here. We will work 
only with homotopy categories $\cK^b_\La$, and only with complexes obtained from a stalk complex $\La$ (concentrated 
in degree $0$), by iterated irreducible mutations. In other words, we study a connected component of the 
silting quiver of $\La$ \cite[see Definition 2.41]{Ai} containing the complex $\La$. Actually, we will study the algebras 
not complexes, that is, the endomorphism algebras of these complexes. \medskip   

For a (weakly) symmetric algebra $\La$, the silting theory becomes slightly more accessible, since then 
all silting complexes are tilting (i.e. no non-zero morphisms $T\to T[i]$, $i\neq 0$). If $\La=KQ/I$ is a symmetric 
algebra, then for any vertex $i\in Q_0$, we have the irreducible mutation $\mu_{P_i}(\La)$ with respect to 
projective $P_i$, which is a tilting complex (as well as $\La$, and again basic). We may consider the endomorphism 
algebra of this complex: 
$$\mu_i(\La):=\End_{\cK^b_\La}(\mu_{P_i}(\La)),$$ 
which we call the mutation of $\La$ at vertex $i$. Since $\mu_i(\La)$ is a tilting complex, the well known results 
of Rickard (see \cite{Rick}) imply that $\mu_i(\La)$ is derived equivalent to $\La$, and hence, it is again a symmetric 
algebra. Moreover, derived equivalencs preserves also periodicity, hence $\mu_i(\La)$ is symmetric and periodic, if 
$\La$ is (see also Section \ref{sec:2}). This motivates to study the mutation class $\mu_\bullet(\La)$ consisting of 
isoclasses of all algebras obtained from $\La$ by iterated mutations. One of the consequences of our main result shows 
that the mutation class of a preprojective algebra coincides with its socle equivalence class (see Corollary 2 below). \medskip 

Our general concern is the study of iterated mutations of a symmetric algebra $\La$ at a fixed vertex $i$, that is, 
the algebras $\mu_i^k(\La)$, for $k\geqslant 1$. Particularly, we want to understand the behaviour of iterations 
$\mu_i^k(\La)$ if $k$ grows, in case the simple module $S_i=P_i/\rad P_i$ is $d$-periodic. \medskip 

The main theorem of this paper relates periodicity of $S_i$ to periodicity of the iterations $\mu_i^k(\La)$ 
of $\La$ at a fixed vertex $i$. Recall that a projective resolution of periodic simple $S_i$ is gives 
rise to an exact sequence of the form 
$$\xymatrix{0 \ar[r] & S_i \ar[r]^{d^0} & P_i \ar[r]^{d^1} & P_i^1 \ar[r]^{d^2} & \dots \ar[r] & 
P_i^{d-2} \ar[r]^{g} & P_i \ar[r] & S_i \ar[r] & 0},$$ 
where $d$ is the period of $S_i$. The main result of this 
article proves periodicity of $\mu_i$ (up to socle equivalence), under an additional assumption on its projective 
resolution. 

\begin{mthm} If $\La=P_i\oplus Q$, and all modules $P_i^k$ belong to $\add Q$, then 
$\mu_i^{d-2}(\La)\sim\La$. \end{mthm} \smallskip 

We will also derive a few interesting consequences of the above theorem. As a first one, we get the following 
immediate corollary. \medskip 

{\bf Corollary 1.} {\it Let $\La$ be a symmetric algebra of period $4$. If $i$ is a vertex of $Q_\La$ and there is 
no loop at $i$, then $\mu_i^2(\La)\sim\La$.} \medskip 

In particular, for any symmetric algebra $\La=KQ/I$ of period $4$, whose Gabriel quiver 
does not contain loops, we have $\mu_i^2(\La)\sim\La$, for all $i\in Q_0$. The following corollary is the next 
consequence of the Main Theorem, applied now in the smallest possible case $d=3$. \medskip 

{\bf Corollary 2.} {\it If $\La=KQ/I$ is a symmetric periodic algebra of period $3$, then $\mu_i(\La)\sim\La$, for all vertices 
$i\in Q_0$ without loop. If $Q$ has no loops, then the mutation class $\mu_\bullet(\La)$ consists of the 
algebras socle equivalent to $\La$.} \smallskip 

In particular, if the field $K$ is of characteritic two, then socle equivalence can be replaced by isomorphism. In 
the proof, we need only periodicity of simples (which is forced by periodicity of the algebra), so we can formulate 
a bit stronger consequence: every twisted $3$-periodic algebra $\La$ having no loops in its Gabriel quiver satisfies 
$\mu_i(\La)\sim\La$. In particular, in characteristic $\neq 2$, it means that the mutation class $\mu_\bullet(\La)$ 
of $\La$ consists only of the isoclass of $\La$. Similar result is well known for preprojective algebras, however all 
of them have periods $2$ or $6$, which seems to be a peculiar example, when period of $\mu_i$ is $1$, but it is 
not a consequence of period of simples. \medskip 

Following the proof of the Main Theorem, we observed that the claim can be proved under slightly weaker assumptions. 
Namely, we do not need to assume that the `middle terms' $P_i^k$ are in $\add Q$. It is sufficient to assume 
that $P_i$ is $\add Q$-periodic (of period $m$), that is, there exists an exact sequence of the form 
$$\xymatrix{P_i \ar[r]^{f^1} & P_i^{(1)} \ar[r]^{f^2} & \dots \ar[r]^{f^m} & P_i^{(m)} \ar[r]^{g} & P_i}$$ 
such that $g$ is a right $\add Q$-approximation, and maps $f^k:P_i^{(k-1)}\to P_i^{(k)}$, for $k\in\{1,\dots,m\}$, 
induce the following left $\add Q$-approximations  
$$(0,\dots,f^k):(\xymatrix{P_i \ar[r]^{f^1} & P_i^{(1)} \ar[r]^{d^2} & \dots \ar[r]^{f^{k-1}} & 
\underline{P_i^{(k-1)}} })\to P^{(k)}_i$$
in the homotopy category (note that the domain is a complex, with degree zero underlined). Then the same arguments show 
that $\mu_i^m(\La)\sim\La$. We mention that if the simple module $S_i$ is periodic of period $d$, with $P_i^k\in\add Q$, 
then $P_i$ is $\add Q$-periodic of period $m=d-2$. A direct corollary from the proof of the Main Theorem gives the following 
result. \medskip 

{\bf Corollary 3.} {\it If $P_i$ is an $\add Q$-periodic module of period $m$, then $\mu_i^m(\La)\sim\La$.} \medskip 

Finally, we will apply the Main Theorem for the special class of algebras $\La$, called the weighted surface algebras. 
We recall that these algebras were introduced and investigated in \cite{WSA,WSAG,WSAC}, where the technical details 
can be found (a bit more in Section \ref{sec:cor4}). We only note that weighted surface algebras are tame symmetric 
algebras of period $4$, except few cases, called the exceptional algebras (see \cite[Section 3]{WSAG} for details) 
and studying their mutations is an important task towards the classification of all tame symmetric algebra of 
period $4$. Indeed, it has been shown \cite{AGQT} that weighted surface algebras exhaust all tame symmetric algebras of 
period $4$ except one family, in case the Gabriel quiver of the algebra is $2$-regular (two arrows start and end at 
any vertex). The remaining exotic family consists of the so called higher tetrahedral algebras \cite{HTA}, which are 
`higher' versions of one of the exceptional algebras. Similarly, we have another exotic family of the so called higher 
spherical algebras \cite{HSA}, and both the higher tetrahedral and higher shperical algebras turned out to be mutations of 
weighted surface algebras (cf. \cite[Section 5]{HSS}). Actually, studying examples of mutations of certain weighted surface 
algebra, we discovered that the second iteration of the mutation at given vertex is always an algebra isomorphic 
to the original one. We wanted to understand what causes this phenomenon, and this led to the results of this paper. \medskip  

{\bf Corollary 4.} {\it Let $\La=KQ/I$ be a weighted surface algebra. Then for any vertex $i\in Q_0$ without loop, we have an 
isomorphism $\mu_i^2(\La)\cong\La$. If there is a loop at $i$, then we have $\mu_i^2(\La)\sim\La$.} \medskip 

The paper is organized as follows. In Section \ref{sec:2}, we give a quick recap through basic concepts of the 
representation theory needed in this paper. Section \ref{sec:mut} is devoted to discuss the definition of the mutation 
an algebra at vertex, while Section \ref{sec:3}, to present the most important facts concerning iterated 
mutations (at given vertex). Then Section \ref{sec:4} provides the proof of the Main Theorem, and also some 
additional considerations about $\add Q$-resolutions and $\add Q$-periodicity, including proofs of Corollaries 
1-3. Finally, in Section \ref{sec:6}, we gives two examples of mutations of a symmetric algebra $\La$ of period $4$, 
which satisfy $\mu_i^2(\La)\sim\La$, and vertex $i$ has a loop (equivalently, assumptions of the Main Theorem are 
not satisfied). Section \ref{sec:cor4} contains a short proof of Corollary 4. \medskip 

For basic background on the representation theory, we refer the reader to the books \cite{ASS,SkY}.  

\bigskip 

\section*{Acknowledgements} The results of this paper were partially presented at the conference 
`Advances in Representation Theory of Algebras (ARTA X)' in Colongne (September 2025). An extended abstract 
of the talk will be published in conference proceedings in LMS Lecture Notes Series (Cambridge University press). 
The author has been supported from the research grant no. 2023/51/D/ST1/01214 of the Polish National Science Center. 

\bigskip 

\section{Basic notions}\label{sec:2} 

For an algebra $\La$, by $\mod\La$ we denote the category of finite-dimensional (right) $\La$-modules 
and by $J_\La$ the Jacobson radical of $\La$. The full subcategory of $\mod\La$ consisting of projective 
(respectively, injective) modules is denoted by $\proj\La$ (respectively, by $\inj\La$). If $M$ is a module 
in $\mod\La$, we write $\add M$ for the additive category of $M$ (that is, the full subcategory of $\mod\La$ 
consisting of modules isomorphic to a direct summand of $M^n$, $n\geqslant 1$). We use the same convention 
for an object $M$ in any abelian category. \medskip 

By a {\it quiver} we mean a quadruple $Q=(Q_0,Q_1,s,t)$, where $Q_0$ is a finite set (of vertices), $Q_1$ a 
finite set (of arrows), and $s,t$ are functions $Q_1\to Q_0$ attaching to every arrow $\alpha\in Q_1$ its 
source $s(\alpha)$ and target $t(\alpha)$. Recall that for a quiver $Q$, its path algebra $KQ$ is the $K$-algebra 
with basis indexed by paths of length $\geqslant 0$ in $Q$ (and multiplication given by concatenation on basis). 
Moreover, the ideal $R_Q$ generated by all paths of length $\geqslant 1$ is the Jacobson radical of $KQ$. 
Note also that, over algebraically closed fields, any algebra $\La$ admits a presentation 
by {\it quiver and relations}, i.e. $\La$ is isomorphic to a quotient $KQ/I$, where $KQ$ is the path algebra 
of a finite quiver $Q$, and $I$ is an {\it admissible} ideal, which means $R_Q^m\subseteq I\subseteq R_Q^2$, 
for some $m\geqslant 2$ (or equivalently, $I$ contains $R_Q^m$ and it is generated by a finite number of so 
called {\it relations}, that is, linear combinations of paths of length $\geqslant 2$ with common source and 
target). The quiver $Q$, for which we have a presentation $\La=KQ/I$, is uniquely determined up to permutation 
of vertices, it is called the {\it Gabriel quiver} of $\La$, and denoted by $Q_\La$. We note that the arrows 
$i\to j$ in $Q_\La$ are in one-to-one correspondence with a $K$-basis of space $e_i J e_j/e_i J^2 e_j$, where 
$J=J_\La=R_Q+I$ is the Jacobson radical of the algebra $\La$. \medskip 

Whenever we consider an algebra $\La$, we assume a fixed presentation $\La=KQ/I$, $Q=Q_\La$. 
Clearly, then the trivial paths $\ve_i\in KQ$ (of length $0$), for $i\in Q_0$, induce a 
complete set of pairwise orthogonal primitive idempotents $e_i=\ve_i+I\in\La$, which sum up to the identity. 
We recall that the modules $P_i:=e_i\La$ form a complete set of pairwise non-isomorphic indecomposable 
modules in $\proj\La$, and each $P_i$, $i\in Q_0$, admits a basis formed by cosets of all paths in $Q$ 
starting from vertex $i$ (modulo $I$). Dually, the left module $\La e_i$ has a basis formed by paths in $Q$ 
ending at $i$. Any arrow $\alpha:i\to j$ (or more generally, an element of $e_i\La e_j$) will be often identified 
with the corresponding homomorphism $\alpha:P_j\to P_i$ in $\proj\La$, given as a left multiplication by $\alpha$. \medskip 

Let $\La$ be an algebra. For a module $M$ in $\mod\La$, we denote by $\Omega(M)$ the syzygy of $M$, that is 
the kernel $\Omega(M):=\ker(p)$ of an arbitrary projective cover $p:P\to M$ of $\mod\La$. Dually, we denote by 
$\Omega^{-1}(M)$ the inverse syzygy of $M$, which is the cokernel $\Omega^{-1}(M):=\cok(e)$ of an arbitrary 
injective envelope $e:M\to E$ of $M$ in $\mod\La$. Both $\Omega(M)$ and $\Omega^{-1}(M)$ do not depend on 
the choice of a projective cover (respectively, injective envelope), up to isomorphism. 
We will sometimes write $\Omega_\La$ or $\Omega^{-1}_\La$, skipping $\La$ if it is clear from the context. 
\smallskip 

By a projective resolution of a module $M$ in $\mod\La$ we mean any exact sequence in $\mod\La$ of the form 
$$\xymatrix{\dots \ar[r]^{d_2} & P_1 \ar[r]^{d_1} & P_0 \ar[r]^{d_0} & M \ar[r] & 0 }$$ 
where for any $n\in\bN$, the epimorphism $P_n\to \ima(d_n)$ induced by $d_n$ is a projective cover 
of $\ima(d_n)=\ker(d_{n-1})$. Projective resolution is uniquely determined up to isomorphism of complexes 
and $\ima(d_n)\simeq\Omega^n(M)$, for all $n\in\bN$. Dually, by an injective resolution of a module $M$ in 
$\mod\La$ we mean any exact sequence in $\mod\La$ of the form 
$$\xymatrix{0 \ar[r]^{} & M \ar[r]^{d^0} & I^0 \ar[r]^{d^1} & I^1 \ar[r]^{d^2} & \dots}$$ 
where for any $n\in\bN$, the monomorphism $\cok(d_{n-1})=I^{n-1}/\ker(d_n)\to I^n$ induced by $d^n$ is an 
injective envelope of $\cok(d_{n-1})$. Similarily, injective resolution is uniquely determined up to 
isomorphism of complexes and $\cok(d_{n-1})\simeq\Omega^{-n}(M)$, for all $n\in\bN$. \medskip

Recall that a module $M$ in $\mod \La$ is said to be {\it periodic} if it is periodic with respect to syzygy, 
i.e. $\Omega^d_\La(M)\simeq M$, for some $d\geqslant 1$. Then minimal such $d$ is called the {\it period} of $M$. 
We will sometimes say that $M$ is $d$-periodic, if $M$ is periodic of period $d$. Note that for any $d$-periodic 
module $M$, we have $P_{n+d}\simeq P_n$, for all $n\geqslant 0$, and moreover, also $\Omega^{-d}(M)\simeq M$, 
so analogous $d$-periodicity holds for the injective envelope. In particular, since projective and injective modules 
coincide over self-injective algebras, one can see that the injective envelope is determined by its first $d$ terms and it 
is given as follows 
$$\xymatrix{0 \ar[r]^{} & M \ar[r]^{d^0} & P_{d-1} \ar[r]^{d^1} & P_{d-2} \ar[r]^{d^2} & \dots \ar[r]^{d^{d-1}}  & 
P_0\ar[r] & \cdots }$$ 
where $I^k=P_{d-1-k}$, $d^{k}=d_{d-1-k}$, except $d^0$ is an embedding $M\simeq Ker(d_{d-1})\to P_{d-1}$. 
We only note that every algebra with periodic simple modules must be self-injective (see \cite{GSS}). 
\smallskip 

An algebra $\La$ is called {\it periodic of period} $d$, provided that $\La$ is $d$-periodic as a bimodule, or 
equivalently, as a (right) module over the enveloping algebra $\La^{\op}\otimes_K\La$. Periodicity 
of an algebra $\La$ implies periodicity of all indecomposable non-projective modules in $\mod\La$, 
and their periods divide the period of $\La$ \cite[see Theorem IV.11.19]{SkY}. In particular, any simple module $S$ in 
$\mod\La$ is periodic, and therefore, $\La$ is self-injective. \medskip

\bigskip 
 
Given an algebra $\La$, we denote by $\cK^b(\mod\La)$ the homotopy category of bounded complexes of modules in 
$\mod\La$ and by $\cK^b_\La$ its (full) subcategory $\cK^b(\proj\La)$ formed by bounded complexes of projective 
modules. The derived category $\cD^b(\mod\La)$ of $\La$ is the localization of $\cK^b(\mod\La)$ with respect to 
quasi-isomorphisms, and admits the structure of a triangulated category, where the suspension functor is given 
by left shift $(-)[1]$; see \cite{Hap}. \smallskip 

Recall that also $\cK^b_\La$ admits the structure of a triangulated category, where the suspension functor is 
the left shift $(-)[1]$. Moreover, it is known that the `canonical' triangulated structure on 
$\cK^b_\La$ is given by the set of distinguished triangles isomorphic to the triangles of the form 
$$\xymatrix{X \ar[r]^{f} & Y \ar[r] & C(f) \ar[r] & X[1]}$$ 
where $f:X\to Y$ is an arbitrary morphism in $\cK^b_\La$ and $C(f)$ is its {\it cone}, that is the complex 
$C(f)=(\xymatrix{C_i\ar[r]^{d^C_i} & C_{i+1}})_{i\in\bZ}\in\cK^b_\La$ given by the modules 
$C_i=X_{i+1}\oplus Y_i=X[1]_i\oplus Y_i$ and the following differentials 
$$d^C_i=\left[\begin{array}{cc} d^X_{i+1} & \\ f_{i+1} & d^Y_i \end{array}\right]$$
induced from the morphism $f=(f_i:X_i\to Y_i)$ and the differentials of 
$X=(\xymatrix{X_i\ar[r]^{d^X_i} & X_{i+1}})$ and $Y=(\xymatrix{Y_i\ar[r]^{d^Y_i} & Y_{i+1}})$; see for example 
\cite{Shah}. Clearly, the remaining map defining the triangle are the cannonical inclusion $\vec{0\\1}: Y\to C(f)$ 
and the cannonical projection $[1 \ 0]:C(f)\to X[1]$. \smallskip  

Two algebras $\La$ and $\La'$ are said to be {\it derived equivalent}, provided that their derived categories 
$\cD^b(\La)$ and $\cD^b(\La')$ are equivalent as triangulated categories. Thanks to the results of Rickard 
from \cite[see Theorem 6.4]{Rick}, we have the following criterion for derived equivalence of algebras. 

\begin{thm}\label{thm:Rick1} Algebras $\La$ and $\La'$ are derived equivalent if and only if $\La'$ is 
of the form $\Lambda'\cong\End_{\cK^b_\La}(T)$, for a complex $T\in \cK^b_\La$ satisfying the following 
two conditions: \begin{enumerate}
\item[(1)] $\Hom_{\cK^b_\La}(T,T[i])=0$, for all $i\neq 0$, 
\item[(2)] $\add T$ generates $\cK^b_\La$ as a triangulated category.  
\end{enumerate} \end{thm} 

Following \cite{Rick}, a complex $T\in\cK^b_\La$ satisfying conditions (1)-(2) in the above theorem is called 
the {\it tilting complex} (over $\La$). In other words, the derived equivalence class of an algebra $\La$ consists 
of all endomorphism algebras of tilting complexes over $\La$. \smallskip 

If we replace $i\neq 0$, by $i>0$ in condition (1), we obtain the definition of a silting complex. We only note 
\cite[see Example 2.8]{Ai} that over symmetric algebras, the notions of silting and tilting complex coincide. \medskip 

Recall also that the derived equivalence preserves symmetricity, as the following result shows \cite[Corollary 5.3]{Rick2}. 

\begin{thm}\label{der:symm} Let $\La$ and $\La'$ be derived equivalent algebras. Then $\La$ is symmetric if and 
only if $\La'$ is symmetric. \end{thm} \smallskip 

Let us also mention that two symmetric algebras $\La,\La'$ are derived equivalent if and only if they are {\it stable 
equivalent}, that is, their stable module catgories $\underline{\mod}\La$ and $\underline{\mod}\La'$ are 
equivalent. Note also that for a symmetric algebra $\La$ its stable module category is isomorphic (as a 
triangulated category) with the homotopy category $\cK^b_\La$. \medskip  

Symmetricity of an algebra $\La$ guarantees that mutation defined in the next section is an algebra derived 
equivalent to $\La$. \smallskip 

It is known since late 80's that periodicity of an algebra is also preserved (together with period), due to 
the following result \cite[see Theorem 2.9]{ES08}. 

\begin{thm}\label{der:per} Let $\La,\La'$ be derived equivalent algebras. Then $\La$ is periodic if and only if 
$\La'$ is periodic. Moreover, if this is the case, then both have the same period. \end{thm} \smallskip 

\begin{rem}\normalfont We eventually mention that derived equivalence also preserves representation type 
in the class of self-injective algebras, i.e. if $\La,\La'$ are derived equivalent self-injective algebras, 
then $\La$ is tame (respectively, wild) if and only if $\La'$ is tame (wild). Moreover, the type of growth 
(domestic, polynomial or non-polynomial) is also preserved. For the proofs we refer the reader to \cite[see Section 2]{HSS}.   
\end{rem} 

\bigskip

\section{Definition of the mutation}\label{sec:mut} 

In this section, we discuss the definition of the mutation of an algebra $\La$ at vertex $i$ of its Gabriel quiver, 
which is precisely the endomorphism algebra of an irreducible silting mutation of the stalk complex $\La$ 
with respect the (indecomposable) summand $P_i$. We also define related notion of mutation class formed by 
all endomorphism algebras of silting complexes reachable from $\La$ by iterated irreducible mutations. \smallskip 

Fix an algebra $\La=KQ/I$ and a module $M$ in $\mod\La$. For a module $X$, we say that a homomorphism 
$f:X\to M'$ is a {\it left $\add M$-approximation of $X$} if $M'$ belongs to $\add M$ and for any homomorphism 
$h:X\to M''$ with $M''$ in $\add M$, there is a homomorphism $h':M'\to M''$ such that $h=h'f$. Equivalently, 
$f:X\to M'$ is a left $\add M$-approximation if and only if for any $M''$ in $\add M$, the induced map 
$$\Hom_A(f,M''):\Hom_\La(M',M'')\to\Hom_\La(X,M'')$$ 
is surjective. Dually, there is a notion of right $\add M$-approximation of a module $X$. \smallskip 

The notion of a left (right) approximation works for any abelian category and its objects $X,M$. In this 
paper, we will use it also for the homotopy categories $\cK^b_\La$. If the category is triangulated (such 
as $\cK^b_\La$), we have a stronger notion of a left (right) $\add\langle M\rangle$-approximation of an 
object $X$. Specifically, if $X,M$ are objects of a triangulated category, say $\cK^b_\La$, then a morphism 
$f:X\to M'$ is called a left $\add\langle M \rangle$-approximation (of $X$), if any morphism $h:X\to M''[i]$, 
for $M''\in\add M$ and $i\in\bZ$, factorizes through $f$. The notions of $\add M$- and 
$\add\langle M\rangle$-approximations may differ in general, but in this paper, we can assume they coincide 
(see the proof of Theorem \ref{thm:}). \medskip  

Now, we recall the definition of a silting mutation, due to Aihara \cite{Ai}. Note that a complex in 
$\cK^b_\La$ is called {\it basic} if it is a direct sum of $n=|\La|$ pairwise non-isomorphic indecomposable 
objects. We first recall the definition of a silting mutation. \medskip 

We will work only with silting objects, instead of subcategories. Actually, for an algebra $\La$ the set 
of all silting subcategories of $\cK^b_\La$ (in the sense of \cite[Definition 2.1]{Ai}) can be regarded 
as the set $\silt(\La)$ of isoclasses of silting complexes \cite[see also Proposition 2.21]{Ai}. Moreover, 
$\La$ is always a silting complex, so every silting subcategory $\cM$ has an additive generator 
$M$ such that $\cM=\add M$, and $M$ is a silting complex \cite[see Proposition 2.20]{Ai}. In this way, 
we identify silting subcategories with their unique (up to isomorphism) basic generator. \smallskip 

Suppose that $T = X\oplus Y\in\cK^b_\La$ is a basic silting complex. Then consider any triangle 
$$\xymatrix{X\ar[r]^{f} & Y'\ar[r]^{g} & X'\ar[r] &X[1]}$$ 
where $f$ is a left $\add Y$-approximation of $X$. Then $\mu_X(T):=X'\oplus Y$ is a silting complex called 
the {\it left mutation of $T$ with respect to $X$}. There is a dual notion of right mutation. Such mutations 
are called irreducible if $X$ is indecomposable. \medskip 

\begin{rem} \normalfont With the above notation the additive category $\add \mu_X(T)=\add X'\oplus Y$ is the same 
as $\add Y\cup\{M'; \ M\in \add(T)\}$, which is exactly the left mutation 
$$\mu^+(\add(T);\add(Y))$$ 
of a silting subcategory $\add T$ (with respect to a subcategory $\add Y$) in the sense of 
\cite[Definition 2.30]{Ai}. In particular, $\add \mu_X(T)$ does not depend on the choice of left 
$\add Y$-approximation. \end{rem} \medskip 

Note that \cite[see Corollary 2.28]{Ai} the total number of non-isomorphic indecomposable direct summands of $T$ 
and $\mu_X(T)$ stays the same, hence in particular, if $T$ is basic, then so is $\mu_X(T)$. This means that indecomposable 
$X'$ is defined up to isomorphism, and it is the cone $X'\simeq C(f)$ of some left $\add Y$-approximation $f$. \smallskip 

With this, take a vertex $i$ of the Gabriel quiver $Q_\La$ of $\La$ and consider $\La=X\oplus Y$, $X=P_i$, 
as a stalk complex. Then $\La$ is a basic silting complex and we may define the {\it mutation of $\La$ at 
vertex $i$} as the endomorphism algebra 
$$\mu_i(\La):=\End_{\cK^b_\La}(\mu_{P_i}(\La))$$ 
of an irreducible silting mutation $\mu_{P_i}(\La)=X'\oplus Y$ of $\La$ with respect to the indecomposable 
direct summand $P_i$. If $\La$ is also symmetric, $\mu_{P_i}(\La)$ is a tilting complex, and hence $\mu_i(\La)$ 
is derived equivalent to $\La$, by Theorem \ref{thm:Rick1}. As a result, also $\mu_i(\La)$ is symmetric, by Theorem 
\ref{der:symm}, and moreover, we have an equivalence of categories 
$$\Hom_{}(T,-): \add\mu_{P_i}(\La)\to \proj \mu_i(\La) .$$ 
By the above remark, the additive category of a silting mutation does not depend on the choice of left 
approximation, and hence, also $\mu_i(\La)$ is a basic symmetric algebra, defined up to isomorphism, with 
the Gabriel quiver $Q_{\mu_i(\La)}$ having the same vertex set as $Q_\La$. \medskip 

This leads to a definition of mutation class. Namely, by the {\it mutation class} of a symmetric algebra $\La$ 
we mean the set $\mu_\bullet(\La)$ of all isoclasses of algebras of the form 
$$\mu_{i_r,\dots,i_1}(\La):=\mu_{i_r}(\cdots    \mu_{i_2}(\mu_{i_1}(\La))\cdots),$$ 
where $(i_r,\dots,i_1)\in Q_0^r$ are arbitrary sequences of vertices of $Q$, $r\geqslant 0$. Note that $\mu_{\bullet}(\La)$ 
is the set of endomorphism algebras of complexes obtained from $\La$ via iterated irreducible silting mutations, 
and it is not the same as the set of endomorphism algebras of all possible iterated mutations of $\La$, i.e. 
also with respect to a decomposable direct summands. Due to recent results of A. Dugas and J. August, we can 
consider notion of the mutation class for a wider context, namely, for any weakly symmetric algebra \cite[see 
Proposition 3.2]{AuDu}. In this paper we will not deal with non-symmetric algebras, however it is an interesting 
further direction. \medskip   

In general, a description of the mutation class is equivalent to a desription of the mutation class of a silting 
complex $\La$, which is a fairly hard task. For silting connected algebras, it is equivalent to the description 
of the set of all silting complexes over $\La$, and not much is known. The main result of this paper shows that 
periodicity of a simple module at vertex $i$ is revealed in periodicity of the mutation at $i$, so it is natural to expect 
similar periodicity at the level of complexes, i.e. iterated mutations of $\La$ at fixed direct summand are finite 
initial subcomplexes of an infinite periodic complex associated to the periodic approximation (see Remark 
\ref{rem:periodicapprox}). \medskip 

Nevertheless, infinitely many silting complexes $\mu_{P_i}^k(\La)$ shall give finitely many endomorphism algebras, 
which is partially confirmed by the Main Theorem. Moreover, we conjecture that for a periodic algebra $\La$ the mutation 
class $\mu_\bullet(\La)$ is finite. In this case all vertices are periodic, and for most of them, we have 
$\mu_i^{d-2}(\La)\cong\La$, so the problem requires first to extend this property for all vertices, and then show that 
mutations at two different vertices commute. Both of these claims are non-trivial, and their proofs are part of an 
ongoing project (work in progress). The first one is slightly touched in this paper (see Sections \ref{sec:4}-\ref{sec:cor4}). 

\begin{rem}\label{rkaprox} \normalfont 

It is worth pointing out that for a (basic) silting complex $T=\La=X\oplus Y$, both 
$X,Y$ are projective and the left $\add(Y)$-approximation can be read of the presentation of the 
algebra $\La=KQ/I$. \smallskip 

Namely, if $X=e_i\La=P_i$, then any left $\add Y$-approximation is determined by a set of paths $p_1,\dots,p_s\in(1-e_i)\La e_i$ 
such that any $a\in e_j\La e_i$ with $e_j$ being a summand of $1-e_i$ factorizes as $a=a_1p_1+\dots+a_sp_s$. 
This is clear when there is no loop at $i$ in $Q$. Then every path $p\in e_j\La e_i$, $e_j$ a summand of $1-e_i$, 
factorizes as $p=p'\alpha$, where $\alpha$ is one of the arrows $\alpha_1,\dots,\alpha_s$ ending at $i$. No loops at 
$i$ forces that every $\alpha_k:i_k\to i$ has $i_k\neq i$, so $e_{i_k}$ is a summand of $1-e_i$. It follows 
that the left $\add Y$-approximation in this case is the map 
$$f=[\alpha_1 \ \cdots \ \alpha_s]^T:P_i\to P_{i_1}\oplus\dots\oplus P_{i_s}=P_i^-,$$ 
given by arrows $p_k=\alpha_k:i_k\to i$, $k\in\{1,\dots,s\}$, ending at $i$ (identified with homomorphisms $P_i\to P_{i_k}$). \smallskip  

If $i$ admits a loop, then the description of a left $\add Q$-approximation becomes more complicated, since we must 
consider maps corresponding to paths of the form $\alpha\rho^p$, where $\alpha:j\to i$ is an arrow ($j\neq i$), and 
$\rho:i\to i$ a loop at $i$. This gets too technical even in case of one loop, so we skip the details. Actually, we 
will not need the precise form of an approximation in case of a loop at vertex. We will need it only in two particular 
examples in Section \ref{sec:6}, where we find the approximations directly, and they involve only paths $\alpha\rho$, 
which gives a relatively easy form. \end{rem}

\section{Iterating mutation at periodic vertex}\label{sec:3} 

Let $\La$ be a (weakly) symmetric algebra and $i$ a vertex of the Gabriel quiver $Q_\La$ of $\La$ 
such that the simple module $S_i$ in $\mod\La$ is $d$-periodic. Then there exists an 
exact sequence in $\mod\La$ of the form 
$$\xymatrix{0 \ar[r]^{} & S_i \ar[r]^{d^0} & P_i \ar[r]^{d^1} & P^1_i \ar[r]^{d^2} & \dots \ar[r]^{d^{d-1}}  & 
P^{d-1}=P_i \ar[r] & S_i\ar[r] & 0 }$$ 
where $Im(d^k)$ is isomorphic to $\Omega_\La^{-k}(S_i)$. \medskip 

The aim of this section is to describe iterated mutations $\mu_i^k(\La)$, $k\geqslant 1$, of $\La$ at a fixed 
periodic vertex $i$. To do this, we construct a sequence of specific (left) approximations in $\cK^b_{\La}$, which 
are related to the above exact sequence. We will show that the two sequences 
coincide under some additional assumption, and parallely discuss the general case. \medskip 

First, we derive the following result, which is an application of \cite[Theorem 4.1]{Du} in a special setup. 
Below we abbreviate $(X,Y):=\Hom_{\cK^b_{\La}}(X,Y)$ for complexes in $\cK^b_{\La}$.  

\begin{thm}\label{thm:} Let $T=X\oplus Q$ be a basic tilting complex in $\cK^b_{\La}$ such that 
$Q$ is a stalk complex concentrated in degree $0$, $X$ is indecomposable (concentrated in degrees 
$\leqslant 0$) and $\mu_X(T)$ is again tilting. Assume 
also that $f:X\to Q'$ is a left $\add Q$-approximation of $X$ in $\cK^b_\La$. Then the following 
statements hold. \begin{enumerate} 
\item[(1)] The induced map $\tilde{f}=(T,f):(T,X)\to (T,Q')$ 
is a left $\add(T,Q)$-approximation of $(T,X)$ in $\proj\tilde{\La}$, where 
$\tilde{\La}=\End_{\cK^b_{\La}}(T)$. 
\item[(2)] For any triangle $\xymatrix@C=0.5cm{X\ar[r]^{f} & Q'\ar[r]^{g} & Y \ar[r] & X[1]}$ 
in $\cK^b_{\La}$, the tilting complex $\tilde{T}=\mu_{(T,X)}(\tilde{\La})$ 
in $\cK^b_{\tilde{\La}}$ satisfies 
$$\End_{\cK^b_{\tilde{\La}}}(\tilde{T})\cong \End_{\cK^b_{\La}}(Y\oplus Q).$$ 
\end{enumerate} 
\end{thm} 

\begin{proof} Recall that the indecomposable projective modules over any endomorphism algebra 
$\End_{\cK^b_\La}(T)$, $T\in\cK^b_\La$, are of the form $(T,T')$, where $T'$ runs through all 
indecomposable direct summands of $T$. In particular, $\tilde{\La}:=\End_{\cK^b_\La}(M)$ is a 
basic algebra if $T$ is basic. In this case, we have an equivalence of categories 
$$\Hom_{\cK^b_\La}(T,-):\add T \to \proj \tilde{\La},$$ 
which sends indecomposable direct summands of $T$ to indecomposable modules in $\proj\tilde{\La}$. 
Moreover, if $T=X\oplus Q$, then $f:X\to Q'$ is a left $\add Q$-approximation of $X$ (in $\add Q$) 
if and only if $(T,f):(T,X)\to (T,Q')$ is a left $\add(T,Q)$-approximation of $(T,X)$ in $\proj\tilde{\La}$. 
This shows the first part of the claim, since $T=X\oplus Q$ is basic, by the assumptions. \medskip 

Now, suppose we have a triangle as in (2). We can may assume that it is the `canonical' triangle with the morphism 
$g:Q'\to Y=C(f)$ identified with its degree zero part $g_0=1_{Q'}:Q'\to Y_0=X[1]\oplus Q'=Q'$ given by the inclusion 
of the second summand of $Y_0$ (actually, the first turns out to be zero). Observe that $g:Q'\to Y$ is a right 
$\add Q$-approximation of $Y$ in $\cK^b_\La$. Indeed, by the assumptions on $X$ and $Y$, we conclude that $Y_i=X_{i+1}$, 
for $i\leqslant -1$, and $Y_0=Q'$, whereas for $i\geqslant 1$, we have $Y_i=0$. This means that $Y=C(f)$ is also 
concentrated in $\leqslant 0$ degrees. Moreover, it follows that any morphism $h:Q''\to Y$, for $Q''\in\add Q$ is 
given by its degree zero part $h_0=g_0h_0:Q''\to Y_0=Q'$, and we have an obvious factorization $h=gh_0$, so $h=gh_0$. 
\smallskip 

Finally note that $g$ is also a right $\add\langle Q\rangle$-approximation of $Y$, because there is no non-zero 
morphism $Q''[i]\to Y$, $i\neq 0$, with $Q''\in\add Q$ (it would imply a non-zero morphism $Q''\to Y[-i]$, which 
is impossible in case $Y\oplus Q$ is tilting). Similarly, one can show that a left $\add Q$-approximation $f:X\to Q'$ 
is simultaneously a left $\add\langle Q\rangle$-approximation of $X$. As a result, the assumptions (a)-(b) of 
\cite[Theorem 4.1]{Du} are satisfied, and (2) follows. \end{proof} \medskip 

We get the following corollary, which helps to control the iterated mutations. 

\begin{cor}\label{cor:mutcone} 
$\mu_{(T,X)}(\End_{\cK^b_{\La}}(X\oplus Q))\cong \End_{\cK^b_{\La}}(C(f)\oplus Q)$, 
where $f:X\to Q'$ is a left $\add Q$-approximtation of $X$ in $\cK^b_{\La}$.  
\end{cor} 

Note that the above result in the case, when $T=\La=X\oplus Q$, $X=P_i$, is equivalent to the definition 
of mutation of $\La$ at $i$, since then $\tilde{\Lambda}=\End_{\cK^b_{\La}}(\La)\cong \La$ 
and $(T,X)\simeq P_i$, so we obtain that 
$$\mu_i(\La)\cong \End_{\cK^b_{\La}}(Y\oplus Q),$$ 
where $f:X\to Q'$ is a left $\add Q$-approximation of $X$ in $\proj\La$ and $Y=C(f)$ is a complex of 
the form $$C(f)=(\xymatrix@C=0.5cm{P_i\ar[r]^{f} & Q'})$$ 
concentrated in degrees $-1$ and $0$.  Now, we want to iterate, so find a mutation $\mu_i^2=\mu_i(\mu_i(\La))$, 
which is isomorphic to the mutation $\mu_{(Y\oplus Q,Y)}(\End_{\cK^b_\La}(Y\oplus Q))$ of the endomorphism 
algebra of $Y\oplus Q$. We will use the following notation 
$$P^{(1)}:=Q', \ f^1:=f \ \mbox{and} \ X_1:=(\xymatrix@C=0.5cm{P_i \ar[r]^{f^1} & P^{(1)} }),$$ 
so that $X_1=Y$. Since the new $T:=Y\oplus Q=X_1\oplus Q$ is again tilting, we conclude from the above 
corollary that $\mu_i^2(\La)=\mu_{(T,X_1)}( \End_{\cK^b_\La}(X_1\oplus Q))\cong \End_{\cK^b_{\La}}(C(F)\oplus Q)$, 
where $F:X_1\to P^{(2)}$ is a left $\add Q$-approximation of $X_1$ in $\cK^b_\La$. Any such $F$ is of the 
form $F=(0,f^2)$, where $f^2:P^{(1)}\to P^{(2)}$ satisfies $f^2f^1=0$ and $P^{(2)}\in\add Q$, and every 
$g:X\to Q''$ with $Q''\in\add Q$ factors through $F$. This shows that the next iteration 
admits an analogous form 
$$\mu_i^2(\La)\cong \End_{\cK^b_\La}(X_2\oplus Q),$$ 
where $X_2\oplus Q$ is tilting and $X_2$ is a complex of the form 
$$C(F)=(\xymatrix@C=0.5cm{P_i \ar[r]^{f^1} & P^{(1)} \ar[r]^{f^2} & P^{(2)}})$$ 
concentrated in degrees $-2,-1$ and $0$. Repeating the above arguments inductively, for $T=X_k\oplus Q$, $k\geqslant 2$, 
we obtain a sequence of tilting complexes $\bP^{(k)}\oplus Q$ in $\cK^b_{\La}$, $k\geqslant 0$, with $\bP^{(k)}$ are 
of the form 
$$\bP^{(k)}=(\xymatrix@C=0.5cm{P_i \ar[r]^{f^1} & P^{(1)} \ar[r]^{f^2} & \cdots \ar[r]^{f^k} & P^{(k)}})$$ 
satisfying the following conditions 
\begin{itemize} 
\item $\bP^{(0)}=(P_i)$ is concentrated in degree $0$, 
\item $\bP^{(k)}$ is concentrated in degrees $-k,\dots,-1$ and $0$, for $k\geqslant 1$, 
\item for every $k\geqslant 1$, the map $f^{(k)}:=(0,\dots,f^k):\bP^{(k-1)}\to P^{(k)}$ (given in degree $0$) is a 
minimal left $\add Q$-approximation of $\bP^{(k-1)}$ in $\cK^b_{\La}$, and 
$$\mu_i^k(\La)\cong \End_{\cK^b_{\La}}(\bP^{(k)}\oplus Q)$$ 
\end{itemize} \medskip 

Note also that for all $k\geqslant 0$, the next tilting complex $\bP^{(k+1)}$ is the cone $C(f^{(k)})$ 
of a left approximation of the previous one. In particular, we obtain simple description of all silting 
complexes reachable from $\La$ via iterated mutation with respect to $P_i$, and these are exactly the 
complexes of the form 
$$\mu_{P_i}^k(P_i\oplus Q)=\bP^{(k)}\oplus Q,$$ 
of theoretically arbitrary length. We will see, that in some cases, periodicity of vertex implies periodicity 
of the complexes $\bP^{(k)}$, i.e. they are subcomplexes of an infinite periodic complex (see Remark \ref{rem:periodicapprox} 
for details). \medskip 

Note that all the above arguments work in case $\La$ is symmetric (weakly), since then all silting complexes 
are tilting (see \cite[Example 2.8]{Ai} or \cite[A.4]{A1}), and the mutation of a silting complex is again silting. So in our setup, complex 
$\La$ and its mutations $\mu_{P_i}^k(\La)$ are always tilting, and Theorem \ref{thm:} applies. \smallskip  

Summing up, we can realize iterated mutations at vertex $i$ as an endomorphism algebras of tilting 
complexes in $\cK^b_{\La}$ whose unique summand not concentrated in degree $0$ is getting larger with 
$k$ increasing. Moreover, the maps $f^k$ are constructed inductively as minimal left approximations of 
previous complexes (equivalently, projective modules over algebras). Motivated by this observation, 
we propose the following definition. 

\begin{defi}\label{cores} \normalfont Consider a complex $\bP\in\cK^b_\La$ of the form 
$$\bP: \qquad 
\xymatrix@C=0.5cm{\underline{P_i} \ar[r]^{f^1} & P^{(1)}_i \ar[r]^{f^2} & \cdots \ar[r]^{f^m} & P^{(m)}_i & }$$ 
which is exact in degrees $1,\dots,m-1$. Denote by $\bP^{(k)}$ the 
complex obtained from the left shift $\bP[k]$ by setting positive degrees to zero. Then $\bP$ is 
called the {\it (left) $\add Q$-resolution of $P_i$} of length $m$, provided that for every $k\in\{1,\dots,m\}$ the 
morphism $$f^{(k)}=(0,\dots,f^k):\bP^{(k-1)}\to P^{(k)}_i$$ 
is a left $\add Q$-approximation of $\bP^{(k-1)}$ in $\cK^b_\La$. \end{defi} \medskip 

We have seen that for any $i$, one can construct a left $\add Q$-resolution of $P_i$, and that it controls 
iterated mutations. In case $\La$ is symmetric, notions of tilting and silting complex coincide, so we could start from 
an arbitrary basic silting complex $T=X\oplus Q$, such that $Q$ is stalk and $X$ indecomposable concentrated in 
$\leqslant 0$ degrees (then both $T$ and $\mu_X(T)=Y\oplus Q$ are always tilting). It is easy to see that such a 
resolution is unique (up to isomorphism of complexes) if we restricted to minimal approximations, i.e. those, for 
which the codomain has a minimal number of indecomposable direct summands. Note that in particular, using Theorem 
\ref{thm:} above, we have constructed a (left) minimal $\add Q$-resolution of $P_i$, since all complexes we dealt 
with were basic. Moreover, observe that the sequence $\bP$ is indeed exact (in degrees $1,\dots,m-1$), because 
at each step $k\in\{1,\dots,m\}$, any morphism $(0,\dots,g):\bP^{(k-1)}\to Q'$ satisfies $gf^{k-1}=0$, so $\ima(f^{k-1}) 
\subset \ker g$, and we get equality if $g$ is a left $\add Q$-approximation of $\bP^{(k-1)}$ in $\cK^b_\La$.   

\bigskip 

\section{Proof of Theorem A}\label{sec:4} 

In this section we prove the first main result of this paper, Theorem A. Fix a (weakly) symmetric algebra $\La$ 
and a vertex $i$ of $Q=Q_\La$ such that $S_i$ is a periodic module of period $d$. In particular, we have 
an exact sequence in $\mod\La$ of the form 
$$\xymatrix{0 \ar[r]^{} & S_i \ar[r]^{d^0} & P_i \ar[r]^{d^1} & P^1_i \ar[r]^{d^2} & \dots \ar[r]^{d^{d-1}}  & 
P^{d-1}=P_i \ar[r] & S_i\ar[r] & 0 }$$ 
with $Im(d^k)$ is isomorphic to $\Omega_\La^{-k}(S_i)$. We additionally assume that all terms $P_i^k$ are 
in $\add Q$, where $\La=P_i\oplus Q$. \medskip 

We shall prove that $\mu_i^{d-2}(\La)\cong \La$. First, observe that the projective-injective coresolution 
of $S_i$ gives rise to the following complex 
$$\bP=\xymatrix@C=1.2cm{P_i \ar[r]^{d^1} & P^{1}_i \ar[r]^{d^2} & \cdots \ar[r]^{d^{d-2}} & P^{d-2}_i,}$$ 
concentrated in degrees $0,1,\dots,d-2$. We claim that $\bP$ is a left $\add Q$-resolution of $P_i$ (of length $d-2$), 
because all $P_i^k$ are in $\add Q$. Indeed, every map $(0,\dots,h):\bP^{(k-1)}\to Q'$, $Q'\in\proj\La$, $k\geqslant 1$, 
is determined by its degree zero part $h:P_i^{k-1}\to Q'$ satisfying $h d^{k-1}=0$. It follows that such a morphism 
$h$ factorizes as $h=h'\pi$, where $\pi:P_i^{k-1}\to \cok(d^{k-1})\simeq\Omega^{k-1}(S_i)$ is the cannonical 
projection. Moreover, $d^k=u\pi$, where 
$u:\cok(d^{k-1})\to P_i^k$ is an 
injective envelope. Now, since $u$ is a monomorphism and $Q'$ is injective, we obtain the following commutative 
diagram  in $\mod\La$ $$\xymatrix{Q' & \\ \cok(d^{k-1}) \ar[u]^{h'} \ar[r]^(.7){u} & P_i^k \ar[lu]_{h''}}$$ 
and consequently, we get that $h=h'\pi= h'' u\pi = h''d^k$ factirizes through $d^k$. As a result, the induced map 
$(0,\dots,d^k):\bP^{(k-1)}\to P_i^k$ is a left $\add Q$-approximation of $\bP^{(k-1)}$ in $\cK^b_\La$, and the 
above sequence is a left $\add Q$-resolution of $P_i$, in the sense of Definition \ref{cores}. \medskip  

Therefore, we conclude from the arguments presented in Section \ref{sec:3} that for any $k\geqslant 1$ 
there is an isomorphism $\mu_i^k(\La) \cong \End_{\cK^b_\La}(\bP^{(k)}\oplus Q)$. In particular, we get 
$$\mu_i^{d-2}(\La)\cong \End_{\cK^b_\La}(T),$$ 
where $T=\mu_{P_i}^{d-2}(\La)=\bP^{(d-2)}\oplus Q$ and $\bP^{(d-2)}$ is a complex in $\cK^b_\La$ of 
the form 
$$(\xymatrix@C=1.2cm{P_i \ar[r]^{d^1} & P^{1}_i \ar[r]^{d^2} & \cdots \ar[r]^{d^{d-2}} & P^{d-2}_i})$$ 
concentrated in degrees $0,-1,\dots,-(d-2)$. Denote by $A$ the endomorphism algebra $A=\End_{\cK^b_\La}(T)$, 
which means that $A\cong \mu_i^{d-2}(\La)$, and let $T=T_1\oplus\dots\oplus T_n$ with $T_i=\bP^{(d-2)}$ and 
$T_k=P_k$, for $k\neq i$. \medskip 

We claim first that $\La$ and $A$ are socle equivalent. Recall that $A=\tilde{P}_1\oplus\dots\oplus\tilde{P}_n$, 
where $\tilde{P}_k=\Hom_{\cK^b_\La}(T,T_k)$, for $k\in\{1,\dots,n\}$, form a complete set of indecomposable 
projective modules in $\mod A$. Obviously, we have an isomorphism $\Lambda\cong\End_{\cK^b_\La}(\La)$, hence 
it is sufficient to construct an isomorphism $$\Phi:B\to A$$ 
(modulo socle), where $B=\End_{\cK^b}(\La)$; in fact, we will see that only 
$\tilde{P}_i/\soc(\tilde{P_i})\simeq P_i/\soc(P_i)$, whereas for other $k\neq i$, there is a stronger 
isomorphism $\tilde{P}_k\simeq P_k$ (preserving multiplication). \medskip 
 
Observe that elements of $B$ can be identified with matrices of morphisms $f:P_k\to P_l$, $k,l\in\{1,\dots,n\}$, 
between the indecomposable direct summands of $\La=P_i\oplus Q$. Similarly, elements of $A$ can be viewed 
as matrices of morphisms $f:T_k\to T_l$, $k,l\in\{1,\dots,n\}$, between the indecomposable direct summands 
of $T=T_i\oplus Q$. Defining $\Phi$ is then reduced to a map sending homomorphisms $P_k\to P_l$ (in $\proj\La$) 
to morphisms $T_k\to T_l$ in $\cK^b_\La$. \medskip 

Now, let $f\in B$ be an arbitrary element of the form $f:P_k\to P_l$, $k,l\in\{1,\dots,n\}$. It is clear 
that $\Phi$ acts as identity: $\Phi(f)=f:T_k\to T_l$, if $k,l\neq i$. If $k=i$ but $l\neq i$, then 
$f:P_i\to P_l$ can be composed with $d_+:=d^{d-1}:P_i^{d-2}\to P_i$ yielding an induced morphism 
$$\Phi(f)=(0,\dots,fd_+):T_i\to T_l.$$ 
(note that morphisms $T_i\to T_l$ in $\cK^b_\La$ can be identified with morphisms of complexes, since 
the homotopy class of the zero morphism $(0,0,\dots,0)$ is a singleton). \smallskip 

On the other hand, if $k\neq i$ and $l=i$, then $f:P_k\to P_i$ is in the radical $\rad_\La$, so $\ima(f)\subseteq 
\rad P_i$, and hence, we get the following commutative diagram in $\mod\La$ 
$$\xymatrix{ & P_k \ar[d]^{\bar{f}} \ar[ld]_{g} \\ P_i^{d-2}\ar[r]^{\bar{d}} & \rad P_i }$$
where $\bar{f}$ and $\bar{d}$ are induced, respectively, from $f$ and $d_+$ (in fact, $\bar{d}$ is a projective 
cover of $\rad P_i$). One can see that if $g':P_k\to P_i^{d-2}$ is another homomorphism making the above diagram 
commutative, i.e. $\bar{d}g'=\bar{f}$, then $d^+(g-g')=0$, so $g-g':P_k\to P_i^{d-2}$ factorizes through 
$\ker(d_+)=\ima(d^{d-2})$, and therefore, we get a factorization $g-g'=d^{d-2}s$, for some $s:P_k\to P_i^{d-3}$. 
This shows that the induced morphism $(0,\dots,g-g'):T_k\to T_i$ of complexes is homotopic to the zero 
morphism. Consequently, we may set 
$$\Phi(f)=[(0,\dots,g)]_\sim:T_k\to T_i,$$ 
and $\Phi(f)$ does not depend on the choice of map $g$. \medskip 

Summing up, we have defined $\Phi(f):T_k\to T_l$, for any $f:P_k\to P_l$, if $(k,l)\neq (i,i)$. 
It is easy to show that the induced map $\Phi_{kl}:\Hom_\La(P_k,P_l)\to\Hom_{\cK^b_\La}(T_k,T_l)$ is a 
$K$-linear epimorphism, for any $(k,l)\neq (i,i)$, and obviously, $\Phi_{k,l}$ is isomorphism for $k,l\neq i$. If 
$\Phi(f)=0$ for $f:P_k\to P_i$, $k\neq i$, then $f=d_+g$ and $(0,\dots,0,g):T_k\to T_i$ is homotopic to zero. 
It means that $g$ factorizes as $g=d^{d-2}s$, for some $s:P_k\to P_i^{d-3}$, and we get $f=d_+g=d_+d^{d-2}s=0$. 
If $\Phi(f)=0$ for $f:P_i\to P_l$, $l\neq i$, then $(0,0,fd_+):T_i\to T_l$ is homotopic to zero, or equivalently, 
we have $fd_+=0$. Consequently, $f$ factorizes as $f=sd_0$, for some $s:S_i\to P_l$, where $d_0:P_i\to S_i$ is the 
cannonical projection onto cokernel of $d_+=d^{d-1}$. It follows that $s=0$, since 
otherwise $s$ is a non-zero monomorphism and $P_l$ has a simple sumbodule $N\simeq S_i$, and then 
$N=\soc(P_l)\simeq S_l$, a contradiction. Hence, also $f=sd_0=0$, and therefore, the map $\Phi_{kl}$ is a monomorphism, 
so an isomorphism, for all $k,l$ (for which it is defined). This extends to an isomorphisms 
$P_k\simeq \Hom_\La(\La,P_k)\to^{\Phi} \Hom_{\cK^b_\La}(T,T_k)=\tilde{P}_k$, for any $k\neq i$. 
Moreover, it is straighforward to check that $\Phi$ preserves compositions, whenever it is defined (on 
the maps being composed, and on the composition as well). \medskip    

Finally, consider $k=l=i$, that is, $f$ is a homomorphism of the form $f:P_i\to P_i$. We want to associate 
a morphism $\Phi(f):T_i\to T_i$, which is a homotopy class of a morphism of complexes $T_i\to T_i$. It 
is natural to consider a resolution of $f$ with respect to the projective resolution of $S_i$, and this 
gives the following commutative diagram 
$$\xymatrix{P_i\ar[r]^{\omega}\ar[d] & P_i\ar[r]^{d^1} \ar[d]^{g^0} & P_i^1\ar[r]^{d^2} \ar[d]^{g^1}& \dots \ar[r]^{d^{d-2}} 
 & P_i^{d-2} \ar[d]^{g^{d-2}}\ar[r]^{d_+} & P_i \ar[d]^{f} \\ 
P_i\ar[r]^{\omega} & P_i\ar[r]^{d^1} & P_i^1\ar[r]^{d^2} & \dots \ar[r]^{d^{d-2}}& P_i^{d-2} \ar[r]^{d_+} & P_i }$$ 
whose rows are taken from the projective resolution of $S_i$. Note that $\omega$ is the composition 
$\omega=d^0d_0$ of the projective cover $d_0:P_i\to\soc(P_i)\simeq S_i$ and the socle inclusion 
$d^0:\ker(d^1)=\soc(P_i)\to P_i$. In fact, $\omega:P_i\to P_i$ is determined by an element 
$\omega\in e_i\La e_i$ generating the socle $\soc(P_i)=K\omega$ (even a cycle in $Q$, but it is not relevant), 
and then $d_0:P_i\to\soc(P_i)$ is the restriction of $\omega$ to the image (note that in particular, we 
have $d_0(e_i)=\omega$ and $\ker(d_0)=\rad P_i$). \smallskip  
 
One would expect that the correct morphism $\Phi(f)$ is the homotopy class of the restricted map 
$(g^0,g^1,\dots,g^{d-2}):T_i\to T_i$, but this does not behave well under the homotopy, namely, the 
map $g=(g^0,g^1,\dots,g^{d-2})$ is not necessarily homotopic to zero, if $f=0$ (or equivalently, 
$[(g^0,\dots,g^{d-2})]_\sim$ may depend on the choice of the resolution of $f$). Nevertheless, this 
is solved by passing to socle quotient. \smallskip 

Observe that $\omega$ gives the socle generator $\omega:P_i\to P_i$ of $\Hom_\La(\La,P_i)\simeq P_i$, and additionally, 
the class $[(\omega,0,\dots,0)]_\sim:T_i\to T_i$ is a generator of the socle $\soc(\tilde{P}_i)$ of $\tilde{P}_i$. 
\smallskip    

Now, if $f=0$ then there is a morphism $G=(\dots,g^0,g^1,\dots,g^{d-2},0)$ of projective resolutions, 
which is homotopic to zero, by standard homological algebra. Moreover, we have $d_+g^{d-2}=fd_+=0$, 
thus $\ima(g^{d-2})\subset\ker(d_+)=\ima(d^{d-2})$, and hence $g^{d-2}$ factorizes as 
$g^{d-2}=d^{d-2}s^{d-2}$. We may continue obtainig homomorphisms $s^k:P_i^k\to P_i^{k-1}$, for 
$k\in\{1,\dots,d-2\}$, $P_i^0:=P_i$, and $s^0:P_i\to P_i$ such that 
$$g^0=d^0_0s^0 + s^{1}d^{1}, \ g^{d-2}=d^{d-2}s^{d-2} \mbox{ and }g^k=d^ks^k + s^{k+1}d^{k+1}, \ 
\mbox{for }k\in\{1,\dots,d-3\}.$$  
Note finally that a morphism of complexes $(h^0,\dots,h^{d-2}):T_i\to T_i$ is homotopic to zero if and 
only if there are homomorphisms $s^k:P_i^k\to P_i^{k-1}$, $k\in\{1,\dots,d-2\}$ such that $h^0=s^1d^1$, 
$h^{d-2}=d^{d-2}s^{d-2}$, and $h^k=d^ks^k + s^{k+1}d^{k+1}$, for all $k\in\{1,\dots,d-3\}$. It follows that 
$g\sim 0$, if $d^0s^0=0$, which is the case in particular, when $s^0\in J$, since $d^0_0=\omega$ generates 
the socle of $P_i$. Assume that $s^0\notin e_i Je_i$, and hence $s^0$ is identified with a 
left multiplication by a unit $\lambda\in K$. Then we have the following morphism 
$$(g^0-\lambda\omega,g^1,\dots,g^{d-2}),$$ 
which is homotopic to zero, and we deduce that $g=(g^0,g^1,\dots,g^{d-2})\sim \lambda(\omega,0,\dots,0)$, 
so $[g]_\sim:T_i\to T_i$ is in the socle of $\tilde{P}_i=\Hom_{\cK^b_\La}(T,T_i)$. As a result, we 
conclude that the coset $[g]_\sim+\soc(\tilde{P}_i)$ is zero in $\tilde{P}/\soc(\tilde{P}_i)$. One 
can also see that for any $f\in\soc(P_i)$, still $d_+g^{d-2}=fd_+=0$, so we can construct analogous 
homotopy as in case $f=0$, and get that then also $[g]_\sim:T_i\to T_i$ belongs to the socle of 
$\tilde{P}_i$. This shows that we have an induced homomorphism    
$$\tilde{\Phi}_{ii}:\Hom_\La(P_i,P_i)/\soc(P_i)\to \Hom_{\cK^b_\La}(T_i,T_i)/\soc(\tilde{P_i})$$ 
which sends cosets of maps $f:P_i\to P_i$ to cosets of morphisms $[g]_\sim$, where 
$g=(g^0,\dots,g^{d-2}):T_i\to T_i$ is taken from arbitrary resolution of $f$ with respect to 
projective resolutions (as above). It is easy to prove that $\tilde{\Phi}_{i,i}$ is also an epimorphism. 
If $\tilde{\Phi}(\bar{f})=0$, then $[g]_\sim$ belongs to the socle of $\tilde{P}_i$, and hence 
$g\sim \lambda(\omega,0,\dots,0)$, for some $\lambda\in K$. It follows that $fd_+=d_+g^{d-2}=d_+d^{d-2}s^{d-2}=0$, 
thus $f$ must belong to the socle, that is $\bar{f}=0$, so we conclude that $\tilde{\Phi}_{ii}$ is 
a monomorphism, and hence an isomorphism. \medskip 

Eventually, we are now able to define the isomorphism 
$$\tilde{\Phi}:B/\soc(P_i)\to A/\soc(\tilde{P}_i).$$ 
Clearly, any element $f\in B/\soc(P_i)$ is represented by a matrix $f=[f_{kl}]$ of morphisms 
$f_{kl}:P_k\to P_l$, for $(k,l)\neq i$, and a coset $f_{ii}=\overline{h_{ii}}$ of a homomomorphism 
$h_{ii}:P_i\to P_i$. Then $\Phi(f)$ is represented by a matrix 
$\tilde{\Phi}(f)=[g_{kl}]$, where $g_{kl}=\Phi_{kl}(f_{kl})$, for $(k,l)\neq (i,i)$ and 
$g_{ii}=\tilde{\Phi}(f_{ii})$. It is clear that $\tilde{\Phi}(f\circ f')=\tilde{\Phi}(f)\circ\tilde{\Phi}(f')$, 
for all $f,f'\in B/\soc(P_i)$, and consequently, we proved that algebras $\mu_i^{d-2}(\La)\cong A$ and 
$\La$ are socle equivalent. This completes the proof of the Main Theorem. \medskip 

However, we can obtain only socle equivalence $\mu_i^{d-2}(\La)\sim\La$ in general, we would like to point out 
that it seems very close to an isomorphism. First, it is an isomorphism in the characteristic $\neq 2$. Moreover, 
we have induced isomorphisms $P_j\simeq \tilde{P}_j$, for $j\neq i$, hence we obtain a bit 
stronger socle equivalence, as stated in the following corollary. 

\begin{cor}\label{strongSE} The quotient algebras $\mu_i^{d-2}(\La)/\soc(\tilde{P}_i)$ and $\La/\soc(P_i)$ are 
isomorphic. \end{cor} \smallskip 

Now, note the following immediate consequence of the Main Theorem, which provides the proof of Corollary 1 (see the 
introduction).  

\begin{cor}\label{cor:sp4} Suppose $\La$ is symmetric and $i$ is a periodic vertex of period 
$d=4$. If there is no loop at $i$ in $Q_\La$, then $\mu_i^2(\Lambda)\sim\Lambda$. \end{cor}  

\begin{proof} If $S_i$ is periodic of period $d=4$, then the exact sequence associated to $S_i$ has the 
form 
$$0\to S_i\to P_i\to P_i^-\to P_i^+\to P_i\to S_i \to 0,$$ 
where modules $P_i^-,P_i^+$ are in $\add Q$ if and only if there is no loop at $i$ in $Q_\La$ (see 
Section \ref{sec:2}). Then the claim follows from the Main Theorem. \end{proof} \medskip 

We hope to strenghten the above result to get an isomorphism in case $i$ admits no loop. This is motivated 
by the property of weighted surface algebras, whose socle equivalence class is formed by the socle deformed 
weighted surface algebras \cite[see Theorem 1.2]{BEHSY}, and non-isomorphic algebras are allowed only if we 
deform socles of projectives $P_i$ such that $i$ admits a border loop (see also Section \ref{sec:cor4}). 
Hence, for vertices without loop, we should expect an isomorphism, not a socle equivalence, but at the 
moment it is not clear how to prove it in general. \medskip 

Note also the following corollary. 

\begin{cor} Let $\La=KQ/I$ be a symmetric algebra of period $d=3$. Then for any vertex $i\in Q_0$ without loop, 
we have $\mu_i(\La)\sim\La$. In particular, if $Q_\La$ has no loops, then $\mu_i(\La)\cong\La$, for all vertices $i$.   
\end{cor} 

\begin{proof} For algebras of period $d=3$, every simple module $S_i$ is periodic of period dividing $3$. We claim 
that $S_i$ is $3$-periodic, if there is no loop at $i$. Indeed, we cannot have a 
simple module $S_i$ of period $1$, if $i$ admits a loop, because then $\rad P_i=\Omega(S_i)$ is simple $S_i$, so 
$P_i$ is a uniserial module with two composition factors $S_i$. But then, there are no non-zero homomorphisms $P_j\to P_i$, 
for $j\neq i$ (see \cite[Lemma I.11.3]{SkY}), and hence, $Q$ consists of one vertex $i$ (since it is connected) and at 
least one loop, a contradiction. As a result, all simple modules $S_i$ in $\mod\La$ at vertices $i$ without loop 
are of period $d=3$, and consequently, by the Main Theorem, we have $\mu_i(\La)=\mu_i^{d-2}(\La)\sim\La$. \end{proof} \medskip 

Note that the above proof of the Main Theorem uses the complex $\bP\in\cK^b_\La$ (induced from a projective 
resolution) to show that it is a (left) $\add Q$-resolution of $P_i$, if summands $P_i^k$ are in $\add Q$. If 
$i$ is arbitrary, we can start with any left $\add Q$-resolution 
$$\bP=\xymatrix@C=0.5cm{\underline{P_i} \ar[r]^{f^1} & P^{(1)}_i \ar[r]^{f^2} & \cdots \ar[r]^{f^m} & P^{(m)}_i},$$ 
of $P_i$, so that also $\mu_i^k(\La)\cong \End_{\cK^b_\La}(\bP^{(k)}\oplus Q)$, for all $k\in\{1,\dots,m\}$. 

Then we can repeat the arguments showing socle equivalence $\mu_i^m(\La)\sim\La$, if only there is a homomorphism 
$d_+:P_i^{(m)}\to P_i$  (in the proof of the Main Theorem, we had $m=d-2$) satisfying the following conditions 
$$\ima(f^{m})=\ker d_+, \ \cok(d_+)=P_i/\rad P_i.$$ \medskip 
Note that in definition of the map $\Theta_{ii}$, we need to use the commutative diagaram whose rows are taken 
from the $\add Q$-resolution (together with $d_+$), instead of the projective resolution. Moreover, we also used 
$\ker(d^1)=\soc(P_i)\simeq S_i$ to get an isomorphism $\tilde{\Theta}_{ii}$ between the socle quotients. In other 
words, the above proof implies the following general consequence. 

\begin{cor} \label{aaprox} Suppose that $\bP$ is a left $\add Q$-resolution 
$$\xymatrix{\underline{P_i} \ar[r]^{f^1} & P^{(1)}_i \ar[r]^{f^2} & \cdots \ar[r]^{f^m} & P^{(m)}_i}$$
such that $\ker(f^1)=\soc(P_i)$. If there is a map $d_+:P_i^{(m)}\to P_i$, such that $\ima(f^m)=\ker(d_+)$ and 
$\cok(d_+)=P_i/\rad P_i$, then $\mu_i^m(\La)$ and $\La$ are socle equivalent. \end{cor} \medskip 

Note that, having a map $d_+:P_i^{(m)}\to P_i$, it has $\cok(d_+)=P_i/\rad P_i$ if and only if 
$d_+$ is a right $\add Q$-approximation. Indeed, if $\cok(d_+)$ is $P_i/\rad P_i$ (equivalently, $\ima(d_+)=\rad P_i$), 
then every map $g:Q'\to P_i$ with $Q'\in\add Q$ is in $\rad_\La$, so $\ima(g)\subset\rad P_i=\ima(d_+)$, and 
we obtain the following commutative diagram  
$$\xymatrix{& Q'\ar[d]_{\bar{g}} \ar[ld]_{g'}\\ P_i^{(m)} \ar[r]^{\bar{d}_+} & \rad P_i,}$$ 
where $\bar{g}$ and $\bar{d}_+$ are maps induced from $g$ and $d_+$ ($Q'$ is projective and $\bar{d}_+$ is an 
epimorphism). This means that $g$ factorizes as $g=d_+g'$, and hence, $d_+$ is a right $\add Q$-approximation. 
Conversly, if $d_+$ is a right $\add Q$-approximation, then it must have $\ima(d_+)=\rad P_i$.  \medskip 

This leads naturally to a notion of $\add Q$-periodicity. Indeed, we call $P_i$ an $\add Q$-periodic module, 
if there exists a {\it periodic $\add Q$-resolution} of $P_i$, that is an exact sequence of the 
form $$(\xymatrix@C=0.5cm{\underline{P_i} 
\ar[r]^{f^1} & P^{(1)}_i \ar[r]^{f^2} & \cdots \ar[r]^{f^m} & P^{(m)}_i \ar[r]^{g}& P_i})$$ 
such that $(\xymatrix{P_i \ar[r]^{f^1} & \dots \ar[r]^{f^m} & P_i^{(m)}})$ is a left $\add Q$-resolution of $P_i$ 
(of length $m$), $\ker(d^1)=\soc(P_i)$ and $g:P_i^{(m)}\to P_i$ is a right $\add Q$-approximation. Then $m$ is 
also called the length of the periodic approximation. Dually, one can define the periodic right $\add Q$-resolution  
of $P_i$ (of length $m$), but this notion is not used in the paper. The smallest such $m$ is called the period 
of $P_i$. \medskip 

Summarizing this section, we have proved that $\mu_i^m(\La)$ is socle equivalent to $\La$ if the 
projective module $P_i$ is $\add Q$-periodic (of period $m$). This proves Corollary 3. As in case of the Main 
Theorem, we have a stronger socle equivalence, namely, an isomorphism 
$$\mu_i^m(\La)/\soc(\tilde{P}_i)\cong\La/\soc(P_i).$$  
In particular, if $S_i$ is a $d$-periodic module with projective resolution satisfying $P_i^k\in\add Q$, then 
a part of this resolution gives a periodic $\add Q$-resolution of length $d-2$, which is a special case. \medskip 

We leave this issue with the following open question. 

\begin{ques}\label{qu1} Let $\La$ be a weakly symmetric algebra. Is it true that, for every $d$-periodic simple module $S_i$ in $\mod\La$, 
the corresponding projective $P_i$ is $\add Q$-periodic (of period $d-2$)? \end{ques} \medskip 

\begin{rem}\label{rem:periodicapprox}\normalfont 
If $\mu_i^m(\La)\cong\La$, for some $m\geqslant 1$, then the mutation 
$\mu_i^{m+1}(\La)$ is the same as $\mu_i(\La)$, so the left $\add Q$-approximation of $(T,\bP^{(m)})\simeq P_i$ in 
$\proj\mu_i^m(\La)\simeq\proj\La$ is given by the following morphism in $\cK^b_\La$ 
$$(0,f^{m+1}):\bP^{(m)}\to P_i^{(1)},$$ 
where $f^{m+1}=f^1d_+$. One can continue obtaining an infinite left $\add Q$-resolution of $P_i$ of the form 
$$\wt{\bP}=\xymatrix{P_i\ar[r]^{f^1} & P_i^{(1)} \ar[r]^{f^2} & \dots\ar[r]^{f^m} & P_i^{(m)} \ar[r]^{f^{m+1}} & P_i^{(1)} 
\ar[r]^{f^2} & P_i^{(2)}\ar[r] & \dots },$$  
which is periodic and encodes all iterated silting mutations of $\La$ at projective $P_i$. That is, every iterated 
mutation $\mu_{P_i}^k(\La)$ has the form $\wt{\bP}^{(k)}\oplus Q$, where $\wt{\bP}^{(k)}$ is identified with a 
finite initial part of $\wt{\bP}$. \end{rem} 

\bigskip 

\section{Some examples}\label{sec:6} 

In this section, we show two examples of iterated mutations of symmetric algebras of period four with 
respect to vertices with loops (see Corollary \ref{cor:sp4}). In particular, we 
will see in both examples, that in fact, periodicity of a vertex implies $\add Q$-periodicity of the 
corresponding projective, which provides a positive answer for the Question \ref{qu1} in this case. 
We will show in the next section that the answer is yes for all weighted surface algebras. \medskip 

Let $Q$ be the following quiver 
$$
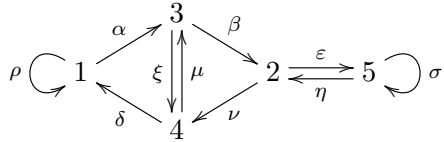
\xymatrix@R=0.3cm{ & 3\ar[rd]^{\beta} \ar@<-.4ex>[dd]_{\xi}&& \\ 
\ar@(lu,ld)[]_{\rho} 1\ar[ru]^{\alpha} & & 2 \ar@<+0.44ex>[r]^{\ve} \ar[ld]^{\nu} & \ar@<+0.44ex>[l]^{\eta} 
5 \ar@(ru,rd)[]^{\sigma} \\ 
 &4\ar[lu]^{\delta} \ar@<-.4ex>[uu]_{\mu} && }$$ 
Then $Q$ is a triangulation quiver associated to a surface, whose triangulation consists of 
two proper triangles $(1 \ 3 \ 4)$, $(4 \ 3 \ 2)$ (with coherent orientation), and a self-folded triangle 
$(2 \ 5 \ 5)$, where $1$ is the unique border edge (see \cite[Section 4]{WSA} for details). This gives a 
permutation $f:Q_1\to Q_1$ of arrows of $Q$, which has the following four orbits (triangles) 
$$(\alpha \ \xi \ \delta), \ (\mu \ \beta \ \nu), \ (\ve \ \sigma \ \eta) \ \mbox{and} \ (\rho).$$ 
Denote by $g$ the permutation $g=\bar{f}:Q_1\to Q_1$, where $\overline{(-)}$ is the involution 
on a $2$-regular quiver $Q$. In this case $g$ has one orbit $(\sigma)$ of length $n_\sigma=1$, one 
orbit $(\xi \ \mu)$ of length $n_\xi=n_\mu=2$ and the remaining orbit 
$$(\alpha \ \beta \ \ve \ \eta \ \nu \ \delta \ \rho)$$ 
of length $n_\alpha=n_\beta=\cdots=7$. Now, take an arbitrary weight function 
$m_\bullet:Q_1\to\bN_{\geqslant 1}$ and a parameter function $c_\bullet:Q_1\to K\setminus\{0\}$, 
which are functions constant on $g$-orbits. We can also consider a border function $b_\bullet$ 
which is given by a single parameter $b=b_{1}$ assocaited to the unique border vertex $1$ of $Q_0$. 
Then $m_\bullet$ and $c_\bullet$ are determined by three weights $m:=m_\alpha\geqslant 1$, 
$n:=m_\xi\geqslant 1$ and $p:=m_\sigma\geqslant 2$, and three parameters $a=c_\alpha$, $c=c_{\xi}$ and $d=c_{\sigma}$, and the associated weighted surface algebra  
$$\La=\La(Q,f,m_\bullet,c_\bullet,b_\bullet)$$ 
is given as a quotient $\La=KQ/I$, where $I$ is generated by the following commutativity 
relations 
$$\alpha\xi=a(\rho\alpha\beta\ve\eta\nu\delta)^{m-1}\rho\alpha\beta\ve\eta\nu, \ 
\xi\delta=a(\beta\ve\eta\nu\delta\rho\alpha)^{m-1}\beta\ve\eta\nu\delta\rho, \ 
\delta\alpha=c(\mu\xi)^{n-1}\mu,$$  
$$\nu\mu=a(\ve\eta\nu\delta\rho\alpha\beta)^{m-1}\ve\eta\nu\delta\rho\alpha, \ 
\mu\beta=a(\delta\rho\alpha\beta\ve\eta\nu)^{m-1}\delta\rho\alpha\beta\ve\eta, \ 
\beta\nu=c(\xi\mu)^{n-1}\xi,$$  
$$\ve\sigma=a(\nu\delta\rho\alpha\beta\ve\eta)^{m-1}\nu\delta\rho\alpha\beta\ve, \ 
\sigma\eta=a(\eta\nu\delta\rho\alpha\beta\ve)^{m-1}\eta\nu\delta\rho\alpha\beta, \ 
\eta\ve=d\sigma^{p-1}$$ 
$$\rho^2=a(\alpha\beta\ve\eta\nu\delta\rho)^{m-1}\alpha\beta\ve\eta\nu\delta+b(\rho\alpha\beta\ve\eta\nu\delta)^m$$ 
and the zero relations of the form $\omega f(\omega) g(\omega)=0$, for all arrows $\omega$ 
except $\omega=\beta$, $\delta$ or $\eta$ (if $n=1$ or $p=2$), and $\omega g(\omega) f(g(\omega))$, for 
all arrows $\omega$ except $\omega=\alpha$, $\nu$ or $\ve$ (if $n=1$ or $p=2$). For more details we refer 
to \cite[Section 5]{WSA} (see also \cite{WSAG} and \cite{WSAC}). We only mention that all commutativity 
relations in $I$ (except the last involving $\rho^2$) have a common form $\theta f(\theta)-c_{\bar{\theta}}A_{\bar{\theta}}$, 
where for an arrow $\theta$, $A_\theta$ denotes the path $\theta g(\theta)\cdots g^{m_\theta n_\theta-2}(\theta)$ 
along the $g$-orbit of $\theta$. Note also that the last relation has a summand of analogous form  
$\rho^2-c_\alpha A_\alpha=\rho f(\rho)-c_{\bar{\rho}} A_{\bar{\rho}}$, but there is an additional 
summand from the socle. This covers all socle deformed weighted surface algebras given by the above 
triangulation quiver (see also \cite[Theorem 1.2]{BEHSY}). \medskip 

One can see that for $n=1$ or $p=2$, the arrows $\xi,\mu$ or $\sigma$ are virtual, so do not 
appear in the Gabriel quiver $Q_\La$ of $\La$. Since we want to consider mutations at vertices 
with loop, we will assume that always $p\geqslant 3$, so that $\rho$ remains in $Q_\La$. 
We may have $n=1$, and then $\La$ has presentation $\La=KQ_\La/I'$, where the Gabriel quiver 
$Q_\La$ is obtained from $Q$ by removing arrows $\xi,\mu$ and generators of $I'$ are 
obtained from generators of $I$ by deleting relations $\delta\alpha=\mu$ and $\beta\nu=\xi$ 
and substitute $\mu:=\delta\alpha$ and $\xi:=\beta\nu$ is all the remaining relations. For 
simplicity, we will work with $n\geqslant 2$, but analogous results can be repeated in case 
$n=1$. \medskip 

We recall that, by general theory \cite[see Theorem 1.3]{WSAG}, algebra $\La$ is symmetric and 
periodic of period $4$; in particular, all simple modules in $\mod\La$ are $d$-periodic 
for $d=4$. Moreover, the algebra $\La$ is tame, but it is not important in this paper. \medskip

Now, we will consider two examples of iterated mutation of $\La$ at vertices $5$ and $1$. Both 
these vertices admit a loop, so this case is not covered by the Main Theorem, but we can always construct 
appropriate periodic $\add Q$-resolution, which provides expected isomorphisms $\mu_i^2(\La)\cong\La$ 
(modulo socle of $P_i$). \medskip 

\begin{expl}\label{ex:6.1}\normalfont First, consider iterated mutations $\mu_5^k(\La)$ of $\La$ at vertex $5$, which 
are controlled by the left $\add Q$-resolution of the projective $P_5$, $\La=P_5\oplus Q$. In this case, 
we have an exact sequence in $\mod\La$ of the form 
$$\xymatrix{0 \ar[r]^{} & S_5 \ar[r] & P_5 \ar[r]^{\vec{\ve \\ \sigma}} & P_2\oplus P_5 \ar[r]^{d^2} & 
P_2\oplus P_5 \ar[r]^{\vec{\eta & \sigma }}  & P_5 \ar[r] & S_5\ar[r] & 0 }$$ 
which determines the (periodic) projective-injective resolution of the simple $S_5$. Note that 
$d^2$ is given by a matrix whose entries can be read off the minimal relations starting (or ending) 
at vertex $5$. Since the modules $P_5^-=P_5^+=P_2\oplus P_5$ have $P_5$ as a direct summand, the 
above sequence does not induce the left $\add Q$-resolution. \smallskip 

It follows from the relation involving $\ve\sigma$, that this path belongs to $\La\ve$, so 
the left $\add Q$-approximation of $P_5$ is a homomorphism of the form $f^{1}=\ve:P_5\to P_5^{(1)}$, 
where $P_5^{(1)}=P_2$ (given as a left multiplication by $\ve$). Similarly, one can check that 
$d_+=\eta:P_2\to P_5$ is a right $\add Q$-approximation of $P_5$ satisfying $\cok(d_+)=P_5/\rad P_5$. \medskip 

Finally, using known basis of projective modules $P_2$ and $\La e_2$ \cite[see Lemma 4.7]{WSAG}, one can see that 
the following sequence  
$$\bP=(\xymatrix{P_5 \ar[r]^{f^1} & P_5^{(1)} \ar[r]^{f^2} & P_5^{(2)}})$$ 
where $P_5^{(2)}=P_2$ and $f^2:P_2\to P_2$ is given as a left multiplication by $\ve\eta$, is 
exact, and moreover, it gives an $\add Q$-resolution of $P_5$ of length $m=2$ (note that $\ve\eta\ve=0$, due 
to the zero relations). It is also easy to check (using relations) that $\ker(f^1)=\soc(P_5)$. \medskip 

As a result, we conclude that $\mu_5(\La)=\End_{\cK^b_\La}(\bP^{(1)}\oplus Q)$ and the second 
iteration $\mu_5^{d-2}(\La)=\mu_5^2(\La)$ is an algebra of the form 
$$\End_{\cK^b_\La}(\bP^{(2)}\oplus Q).$$ 
Since $d_+=\eta:P_2\to P_5$ is a right $\add Q$-approximation for $P_i$ (satisfying $\ima(f^2)=\ker(d_+)$), 
we conclude that it admits a periodic (left) $\add Q$-resolution, and hence $\mu_5^2(\La)\sim\La$ (modulo $\soc(P_5)$), 
due to Corollary \ref{aaprox}. Actually, it is known \cite[see Theorem 1.2]{BEHSY} that $\La$ admits socle deformations only with 
respect to the vertex $1$ (having a border loop), so in this case, we have an isomorphism $\mu_5^2(\La)\cong\La$. \end{expl} \medskip 

\begin{expl}\label{ex:6.2}\normalfont In the second example, we will study iterated mutations $\mu_1^k(\La)$ of 
$\La$ at vertex $1$; $\La=P_1\oplus Q$. In particular, as in previous example, we will show that 
$\mu_1^2(\La)\sim\La$. This case is slightly different, since paths $\rho\alpha$ and $\delta\rho$ 
are not involved in minimal relations, but the general strategy stays the same. \medskip 

As before, we start with the (minimal) projective-injective resolution of $S_1$, which is an exact sequence 
in $\mod\La$ of the form 
$$\xymatrix{0\ar[r] & S_1 \ar[r] & P_1 \ar[r]^(0.35){\vec{\delta \\ \rho}} 
& P_4\oplus P_1 \ar[r] & P_3\oplus P_1 \ar[r]^(0.65){\vec{\alpha & \rho}} & P_1 \ar[r] & S_1 \ar[r] & 0 }$$ 
Similarly, it does not induce the $\add Q$-resolution, since the modules $P_1^-=P_4\oplus P_1$ and 
$P_1^+=P_3\oplus P_1$ are not in $\add Q$.  

Now, using the relation involving $\rho^2$, one can easily see that the map 
$f^1:P_1\to P_1^{(1)}=P_4\oplus P_4$ given as $f^1=\vec{\delta \\ \delta\rho}$ is a left $\add Q$-approximation 
of $P_1$. Exploiting relations and bases of projective modules over $\La$ (and $\La^{\op}$), we can 
check that the following sequence  
$$\bP=(\xymatrix{P_1 \ar[r]^{f^1} & P_1^{(1)} \ar[r]^{f^2} & P_1^{(2)}})$$ 
is exact (in $P_1^{(1)}$), where $P_1^{(2)}=P_3\oplus P_3$ and $f^2$ is given by the following matrix 
$$\vec{\xi & -aA\\ 0 & \xi},$$ 
where $A=(\beta\ve\eta\nu\delta\rho\alpha)^{m-1}\beta\ve\eta\nu$. Moreover, 
one can similarily see that $\bP$ is an $\add Q$-resolution of $P_1$ of length $m=2$. 
In particular, it follows that $\mu_1(\La)=\End_{\cK^b_\La}(\bP^{(1)}\oplus Q)$ and $\mu_1^2(\La)=
\End_{\cK^b_\La}(\bP^{(2)}\oplus Q)$. Now, straighforward calculations show that the map 
$d_+=\vec{\alpha & \rho\alpha}:P_1^{(2)}\to P_1$ is a right $\add Q$-approximation of $P_i$ satisfying 
$\ima(f^2)=\ker(d_+)$ and $\cok(d_+)=P_1/\rad P_1$, andconsequently, we conclude from Corollary \ref{aaprox} 
that again $\mu_1^{d-2}(\La)=\mu_1^2(\La)\sim\La$. In this case, $\mu_1^2(\La)$ has the same quiver and 
relations as $\La$, but with possibly different $b\in K$. \end{expl} \medskip 

Summing up, in the above two examples, we took a weighted surface algebra $\La$, for which all 
simple modules $S_i$ have period $4$, and we considered two iterated mutations of $\La$ with respect 
to a vertex $i$. In both cases, vertex $i$ admits a loop, so its (periodic) projective resolution does 
not satisfy assumptions of the Main Theorem. But still $\mu_i^2(\La)\sim\La$, since in both examples we have 
constructed a periodic (left) $\add Q$-resolution of the associated projective $P_i$. 

\bigskip  

\section{Proof of Corollary 4}\label{sec:cor4}
In the last short section we provide arguments for Corollary 4. Let $\La=\La(Q,f,m,c)$ be a weighted surface algebra 
of infinite representation type (different from the exceptional ones). Then it is known, that $\La$ is periodic of 
period $d=4$, and also all simple modules in $\mod\La$ are periodic of period $4$ (in infinite representation type we 
cannot have simples of period $2$; see \cite{Enote}). Consequently, by Corollary 1, we have $\mu_i^2(\La)\sim\La$, for 
any vertex $i$ without loop. Therefore, it suffices to show that $\mu_i^2(\La)\sim\La$, for a vertex $i$ with loop. 

By the construction (see \cite{WSA,WSAC}), $\La=KQ/I$ is given by the quiver $Q$ being a glueing of a finite number 
of the following three types of blocks 
$$
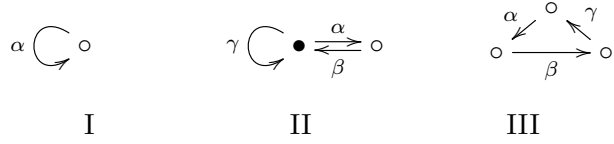
\xymatrix@C=0.3cm@R=0.2cm{\\ \ar@(lu, ld)[]_{\alpha}\circ} \qquad \qquad
\xymatrix@C=0.6cm@R=0.2cm{&\\ \ar@(lu, ld)[]_{\gamma}\bullet\ar@<.35ex>[r]^{\alpha}&\ar@<.35ex>[l]^{\beta}\circ} 
\quad \qquad 
\xymatrix@C=0.3cm@R=0.2cm{&\circ\ar[ld]_{\alpha}&\\ \circ\ar[rr]_{\beta}&&\circ\ar[lu]_{\gamma}}$$ 
$$\mbox{ I }\qquad\qquad\quad\quad\mbox{ II }\qquad\qquad\qquad\mbox{ III }$$ 
corresponding to border edges, self-folded triangles, and triangles of the triangulated surface, respectively. By 
a glueing we mean a disjoint union of blocks, for which we glue every vertex $\circ$ with a vertex $\circ$ contained 
in another block. It is known that $Q_\La$ contains loops of two kinds: border loops (i.e. loops from a block of type I) 
or loops in the self-folded triangles (blocks of type II). We recall that each block gives an orbit of the 
induced permutation $f:Q_1\to Q_1$, which is $(\alpha)$ in type I, or $(\alpha \ \beta \ \gamma)$ in types II and 
III. As in example from Section \ref{sec:6}, we have an associated permutation $g:Q_1\to Q_1$, given as 
$g(\alpha)=\overline{f(\alpha)}$, where $\overline{(-)}$ is the involution for $Q$, and then ideal $I$ is 
generated by the relations of the form $\alpha f(\alpha)-c_{\ba} A_{\ba}$, $c_\alpha\in K^*$ constant on 
$g$-orbits, for all arrows $\alpha\in Q_1$ (except a socle deformed relation for border loops $\alpha$) and 
the zero relations $\alpha f(\alpha)g(f(\alpha))$ and $\alpha g(\alpha) f(g(\alpha))$, for some of the arrows 
$\alpha$ (see \cite[Definition 2.2]{WSAC}).  The paths $A_\alpha$ are of the form 
$\alpha g(\alpha)\cdots g^{m_\alpha n_\alpha-2}(\alpha)$, for weights $m_\alpha\geqslant 1$, also constant 
on $g$-orbits. We will only need relations $\alpha g(\alpha)f(g(\alpha))=0$, for arrows $\alpha$ 
such that $m_{f(\alpha)}n_{f(\alpha)}\geqslant 3$, and relations $\alpha f(\alpha) g(f(\alpha))$, 
for arrows $\alpha$ such that $m_{f^2(\alpha)}n_{f^2(\alpha)}\geqslant 3$. \medskip 

{\it Case 1}. Suppose $i$ is a vertex with a loop $\sigma$ contained in a block of the form 
$$
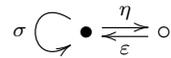
\xymatrix@C=0.6cm@R=0.2cm{\ar@(lu, ld)[]_{\sigma}\bullet\ar@<.35ex>[r]^{\eta}&\ar@<.35ex>[l]^{\ve}\circ}$$ 
By definition, $\La$ admits the following relations 
$$\ve\sigma=c_{\bar{\ve}}A_{\bar{\ve}} \mbox{ and } \sigma\eta=c_{\eta}A_{\eta},$$  
where $c_{\bar{\ve}},c_\eta\in K^*$, $A_{\bar{\ve}}\in\La\ve$ and $A_\eta\in\eta\La$ are paths along $g$-orbits of 
$\bar{\ve}$ or $\eta$, respectively. This shows that any path of the form $\ve\sigma^k$ or $\sigma^k\eta$ belong 
to $\La\ve$ or $\eta\La$, respectively, and consequently, $f^1=\ve:P_i\to P_i^{(1)}=P_j$ is a left $\add Q$-approximation 
of $P_i$, whereas $d_+=\eta:P^{(2)}_i=P_j\to P_i$ is its right $\add Q$-approximation. 

Moreover, we have zero relations in $\La$ of the form $\ve\eta\ve=\ve g(\ve) f(g(\ve))=0$, because 
$f(\ve)=\sigma$ satisfies $m_\sigma n_\sigma\geqslant 3$ (indeed, otherwise $A_\sigma=\sigma$, so 
$\eta\ve=\eta f(\eta)=c_{\bar{\eta}}A_{\bar{\eta}}=c_\sigma A_\sigma=c_\sigma \sigma$ and $\sigma\in J^2$ 
is not an arrow of $Q_\La$). Similarly, we have $\eta\ve\eta=\eta f(\eta) g(f(\eta))=0$, since again 
$f^2(\eta)=\sigma$ satisfies required condition. Using known description of bases for projective $\La$-modules 
\cite[see Lemma 4.7]{WSAG} (and other relations), one can compute as in Example \ref{ex:6.1} that the map 
$f^2=\ve\eta:P_j\to P_j$ gives a left $\add Q$-approximation of $(\xymatrix{P_i\ar[r]^{f^1} & P_j})$ in $\cK^b_\La$, 
and the following sequence is exact 
$$\xymatrix{P_i\ar[r]^{f^1} & P_j\ar[r]^{f^2} & P_j\ar[r]^{d_+} & P_i}.$$ 
It is easy to check that $\ker(f^1)=\soc(P_i)$ and $\cok(d_+)=P_i/\rad P_i$, hence it follows that $P_i$ is 
$\add Q$-periodic (of period $2$). Therefore, by to Corollary \ref{aaprox}, we have $\mu_i^2(\La)\sim\La$, 
and we are done in this case. \medskip 

{\it Case 2.} Now, suppose $i$ admits a loop $\rho$ in a block of type I. In this case, the Gabriel quiver $Q_\La$ 
has a subquiver of the form 
$$\xymatrix{x\ar[r]^{\delta} & i \ar@(ul,ur)[]^{\rho} \ar[r]^{\alpha} & j}$$ 
with $j\neq x$, and we have a $g$-orbit of the form $(\cdots \ \delta \ \rho \ \alpha \ \cdots)$. 
As in Example \ref{ex:6.2}, we have $\rho^2=\rho f(\rho)=c_\alpha A_\alpha$, so $\rho^2$ 
belongs to $\La\delta$, because $A_\alpha$ ends with $g^{m_\alpha n_\alpha-2}(\alpha)=\delta$. 
Clearly, also $\rho^2\in\alpha\La$, so the left $\add Q$-approximation of $P_i$ has the form 
$f^1=\vec{\delta \\ \delta\rho}:P_i\to P_x\oplus P_x$, while the right $\add Q$-approximation is given as 
follows $d_+=[\rho\alpha \ \alpha ]:P_j\oplus P_j\to P_i$. One can show using relations and bases that the 
left $\add Q$-approximation of $ $ in $\cK^b_\La$ is given by the map $f^2:P_x\oplus P_x\to P_j\oplus P_j$ 
defined by the matrix 
$$\vec{\xi & -aA \\ 0 & \xi},$$ 
where $A$ is a path along the $g$-orbit of the form $A=\beta g(\beta)\dots \nu$ (of length $m_\beta n_\beta-3$), 
with $\beta=g(\alpha)$. Hence again, the maps $f^1,f^2$ together with $d_+$ give rise to a periodic 
$\add Q$-resolution. As a result, also in this case we have $\mu_i^2(\La)\sim\La$, by Corollary \ref{aaprox}, 
and the proof of Corollary 4 is now complete.

\end{document}